\documentclass[a4paper,11pt,leqno,english]{smfart}
\usepackage{aeguill}
\usepackage{enumerate}
\usepackage{amssymb,amsmath,latexsym,amsthm}
\usepackage[T1]{fontenc}
\usepackage{geometry}
\usepackage{url}
\usepackage[frenchb, main=english]{babel}
\usepackage[utf8]{inputenc}
\usepackage{mathrsfs}
\usepackage{xcolor}
\usepackage{comment}
\definecolor{violet}{rgb}{0.0,0.2,0.7}
\definecolor{rouge2}{rgb}{0.8,0.0,0.2}
\usepackage{tikz}
\usepackage{empheq}
\usepackage{tikz-cd}
\usetikzlibrary{matrix,arrows,decorations.pathmorphing}
\usepackage{hyperref}
\usepackage{mathpazo}
\hypersetup{
    bookmarks=true,         
    unicode=false,          
    pdftoolbar=true,        
    pdfmenubar=true,        
    pdffitwindow=false,     
    pdfstartview={FitH},    
    pdftitle={},    
    pdfauthor={},     
    colorlinks=true,       
   linkcolor=rouge2,          
    citecolor=violet,        
    filecolor=black,      
    urlcolor=cyan}           
\usepackage{enumitem}
\usepackage{appendix}

 \theoremstyle{plain}    
 \newtheorem{thm}{Theorem}[section]
\theoremstyle{plain} 
\newtheorem{bigthm}{Theorem}

 \numberwithin{equation}{section} 
 \numberwithin{figure}{section} 
 \newtheorem{cor}[thm]{Corollary} 
 \theoremstyle{plain}    
 \newtheorem{prop}[thm]{Proposition} 
 \theoremstyle{plain}    
 \newtheorem{lem}[thm]{Lemma} 
 \theoremstyle{remark}
  \newtheorem{claim}[thm]{Claim} 
 \theoremstyle{remark}
 \newtheorem{rem}[thm]{Remark}
 \theoremstyle{definition}

\theoremstyle{plain}  
\newtheorem{set}[thm]{Setup}
\theoremstyle{plain}
\newtheorem*{quest*}{Main Question}
\theoremstyle{plain}

\theoremstyle{plain}

\theoremstyle{definition}
\newtheorem{defi}[thm]{Definition}

\newtheorem*{ackn}{Acknowledgements}

\newcommand{\C}{{\mathbb{C}}}
\newcommand{\N}{{\mathbb{N}}}

\newcommand{\Q}{{\mathbb{Q}}}
\newcommand{\R}{{\mathbb{R}}}

\newcommand{\Z}{{\mathbb{Z}}}
\newcommand{\bD}{{\mathbb{D}}}

\newcommand{\cC}{{\mathcal{C}}}

\newcommand{\cO}{{\mathcal{O}}}

\newcommand{\cX}{{\mathcal{X}}}

\def\1{\mathbf{1}}

\newcommand{\wh}{\widehat}

\newcommand{\e}{\varepsilon}
\newcommand{\ep}{\varepsilon}

\newcommand{\om}{\omega}
\newcommand{\f}{\varphi}

\newcommand{\p}{\psi}

\newcommand{\omh}{\widehat \omega}

\newcommand{\omt}{\widetilde{\om}_{t}}

\newcommand{\vp}{\varphi}

\newcommand{\Ric}{\mathrm{Ric}}

\newcommand{\reg}{\mathrm{\rm reg}}

\renewcommand{\ge}{\geqslant}
\renewcommand{\le}{\leqslant}

\newcommand{\sing}{\operatorname{\rm sing}}

\newcommand{\cY}{{\mathcal{Y}}}

\title{Strict positivity of K\"ahler-Einstein currents}
\date{\today}

\author{Vincent Guedj}

\email{vincent.guedj@math.univ-toulouse.fr}

\address{Institut de Math. de Toulouse; UMR 5219, Université de Toulouse; CNRS, 
UPS, 118 route de Narbonne, F-31062 Toulouse, France}

\author{Henri Guenancia}
\email{henri.guenancia@math.cnrs.fr}

\address{Institut de Math. de Toulouse; UMR 5219, Université de Toulouse; CNRS, 
UPS, 118 route de Narbonne, F-31062 Toulouse, France}

\author{Ahmed Zeriahi}

\email{ahmed.zeriahi@math.univ-toulouse.fr}

\address{Institut de Math. de Toulouse; UMR 5219, Université de Toulouse; CNRS, 
UPS, 118 route de Narbonne, F-31062 Toulouse, France}

\begin{document}

\begin{abstract}  
K\"ahler-Einstein currents, also known as singular Kähler-Einstein metrics, have been introduced and constructed a little over a decade ago. These currents live on mildly singular compact Kähler spaces $X$ and their two defining properties are the following: they are genuine K\"ahler-Einstein metrics on $X_{\rm reg}$  and they admit local bounded potentials near the singularities of $X$.
In this note we show that these currents 
dominate a Kähler form
near the singular locus, when either $X$ admits a global smoothing,
or when $X$ has isolated smoothable singularities.
Our results   apply to klt pairs and allow us to show that if $X$ is any compact K\"ahler space of dimension $3$ with log terminal singularities, then any singular K\"ahler-Einstein metric of non-positive curvature dominates a K\"ahler form.
\end{abstract} 

\maketitle

\tableofcontents

\section*{Introduction}

\subsection*{Introducting the main problem}
\label{intro}
Einstein metrics are a central object in differential geometry.
A K\"ahler-Einstein metric on a complex manifold is a K\"ahler metric
whose Ricci curvature is proportional to the metric tensor.
A foundational result of Yau  \cite{Yau78} allows one to construct many examples of  
these fundamental objects. 

In connection with the Minimal Model Program,
singular K\"ahler-Einstein metrics  on mildly singular K\"ahler varieties $X$
have been constructed in \cite{EGZ09,BG,BBEGZ} and further studied by many authors
(see \cite{GZbook, Bouck18,DonICM18,ChiLi22} and the references therein).
These are K\"ahler forms $\omega_{\rm KE}=\omega+dd^c \f_{\rm KE}$
on the regular part $X_{\rm reg}$ of $X$, where $c_1(X)=\lambda [\omega]$ is proportional to a reference K\"ahler class
$[\omega]$, such that
$$
\Ric(\omega_{\rm KE})=\lambda \omega_{\rm KE},
$$
and which admit bounded (or mildly unbounded) local potential near the singularities $X_{\rm sing}$.
In particular $\omega_{\rm KE}$ uniquely extends as a positive closed current to $X$.

One constructs $\omega_{\rm KE}=\omega+dd^c \f_{\rm KE}$ by solving a 
complex Monge-Amp\`ere equation
\begin{equation}
\label{eq KE}
\tag{KE}
(\omega+dd^c \f_{\rm KE})^n=e^{-\lambda \f_{\rm KE}} \mu_X,
\end{equation}
where $n=\dim_{\C} X$ and $\mu_X$ is an appropriate volume form.

Due to the presence of singularities, the geometry of these {\it K\"ahler-Einstein currents} is quite mysterious
despite recent important progress \cite{DS,HS17}.
Understanding the asymptotic behavior of $\f_{\rm KE}$ near $X_{\rm sing}$ is a major open problem.

\smallskip

In this note we partially address the following basic question:

\begin{quest*}
Is the Kähler-Einstein metric $\omega_{\rm KE}$ solving \eqref{eq KE} a {\it K\"ahler current}, i.e. does $\omega_{\rm KE}$
dominate a K\"ahler form ? 
\end{quest*}

\noindent
Note that unless $X$ is smooth, $\omega_{\rm KE}$ is never dominated by a Kähler form, cf Corollary~\ref{domination}. \\

\noindent
{\bf Known cases.} Although the question sounds much easier than asking for asymptotics of $\omega_{\rm KE}$, it has not yet been addressed in full generality and only a handful of particular cases seem to be understood, which we briefly survey below.\\

\noindent
$\bullet$ {\it Orbifold singularities}. If $X$ has only finite quotient singularities (i.e. $X$ is an orbifold) then $\om_{\rm KE}$ is a smooth orbifold Kähler metric (i.e. it is a Kähler metric in the local smooth uniformizing charts). In particular, $\om_{\rm KE}$ is a Kähler current. Since $2$-dimensional log terminal singularities are quotient singularities, it follows that the main Question admits a positive answer in dimension two for log terminal singularities. More generally, $\om_{\rm KE}$ is orbifold-smooth near any quotient singularity of $X$ at least when $X$ is projective \cite{LiTian19}; in particular it is a Kähler current on the orbifold locus of $X$. \\

\noindent
$\bullet$ {\it Limits of smooth spaces}. Other examples include when $X$ admits a crepant resolution, or when $X$ can be suitably obtained as limit of smooth Kähler-Einstein manifolds \cite{RZ0, DS, HS17}. The common thread in the situations appearing in {\it loc. cit.} is that one can embed $\iota: X\hookrightarrow \mathbb P^N$ with $\om_{\rm KE}=\om_{\rm FS}|_X+dd^c \varphi$ and approximate $(X, \om_{\rm KE})$ by a sequence of compact Kähler manifolds $(X_k, \om_k)$ such that there are
\begin{enumerate}[label=$\bullet$]
\item[-] embeddings $\iota_k:X_k \hookrightarrow \mathbb P^N$ 
with $\om_k=\om_{\rm FS}|_{\iota_k(X_k)}+dd^c \varphi_k$ and $ \|\varphi_k\|_{L^{\infty}(X_k)} \le C$,
\item[-] uniform Ricci lower bounds $\Ric \, \om_k \ge -C \om_k$
\end{enumerate}
for some uniform constant $C>0$. From these estimates, an easy application of Chern-Lu formula yields $\om_k\ge C^{-1} \om_{\rm FS}|_{\iota_k(X_k)}$ from which the strict positivity of $\om_{\rm KE}$ follows. 

In summary, the ideas above allow to treat the case of singular Kähler-Einstein varieties that are degenerations of {\it smooth} Kähler manifolds with uniform lower Ricci bound and uniform $L^{\infty}$ bound for its potential.\\

Unfortunately, general singular Kähler-Einstein metrics $\om_{\rm KE}$ on a singular space $X$ arise by construction as limits of smooth Kähler metrics $\om_\ep$ on a desingularization $\pi:\wh X\to X$ such that $\Ric_{\om_\ep} \to -\infty$ along any divisor $E\subset \wh X$ with positive discrepancy. Moreover, a general singular variety cannot be obtained as a degeneration of smooth varieties, since there might be (even local and topological) obstructions to smoothability. This prevents one from directly applying the above ideas to a general singular Kähler-Einstein space. This paper grew out as an attempt to find the largest possible field of application of the general technique recalled above. More precisely, the goal of the present paper is threefold: 

\begin{enumerate}[label={\bf \Alph*.}]
\item Formulate a general framework where the above global strategy applies {\it mutatis mutandis}, including non-projective Kähler spaces as well as singular pairs.
\item Provide a local version of the approach that enables to treat isolated smoothable singularities. 
\item Develop a systematic use of the technique in order to enlarge the class of singularities (beyond the smoothable or crepant ones) for which one can answer positively the Main Question above, using a {\it several step} degeneration process.
\end{enumerate}

So far, most (general) results about singular Kähler-Einstein metrics have been derived by establishing uniform a priori estimates on smooth approximants (space and metric). The novelty of our approach is that given a singular space/metric $(X,\om_{\rm KE})$, we are able to answer the Main Question positively for $(X,\om_{\rm KE})$ as soon as one can suitably approximate our space/metric by a {\it possibly singular} space/metric on which we can {\it qualitatively} answer the Main Question, as long as we have a lower bound on the Ricci curvature and a uniform bound on a suitable potential. We refer to the last paragraph of this introduction or \textsection~\ref{threefolds} for an explicit application of this principle.

\subsection*{Statement of the results}
Let us now get a bit more explicit and expand what we mean by the above stated goals. Our first main result is as follows:

\begin{bigthm} \label{thmA}
    Let $X$ be a normal compact Kähler space with log terminal singularities such that $K_X\sim_{\Q}\mathcal O_X$, and let $\alpha$ be a Kähler class. 
    If $X$  is smoothable then the unique singular Ricci-flat metric $\omega_{\rm KE} \in \alpha$ 
    dominates a K\"ahler form.
\end{bigthm}

This strict positivity result actually holds in more general contexts (canonically polarized varieties, $\Q$-Fano varieties,
klt pairs) as we explain in Theorem \ref{thm strict}.
The existence of a {\it global} smoothing is a rather restrictive assumption,
although it holds e.g. at the boundary of the moduli space of positively curved K\"ahler-Einstein metrics.
The latter has a natural compactification arising from the Gromov-Hausdorff topology,
 and the geometric meaning of the boundary points was elucidated in \cite{ DS,SSY}.
 
 \smallskip

An isolated singularity  is more likely to admit a local smoothing.
Shifting perspective we use the local singular theory developed in \cite{GGZ20}
to establish  a positivity result 
of solutions to local Monge-Amp\`ere equations
at isolated smoothable singularities.
As a consequence we obtain our second main result.

  \begin{bigthm}  \label{thmB}
  Let $X$ be a normal compact Kähler space with log terminal singularities such that $K_X\sim_{\Q}\mathcal O_X$, and let $\alpha$ be a Kähler class. 
 The unique singular Ricci-flat metric $\omega_{\rm KE} \in \alpha$ 
    dominates a K\"ahler form near any smoothable isolated singularity.
 \end{bigthm}
 
  The result is more general and we refer the reader to Theorem \ref{thm:positivitylocale} and Corollary~\ref{cor klt pairs}
  for more precise results, including the case of klt pairs with isolated singularities.
  
  \smallskip

  As recalled 
  earlier, the basic idea is not new. Deforming $X$ 
  in a smooth K\"ahler
  approximant $X_t$,  we would like to use Chern-Lu formula \cite{Chern,Lu}
  and establish a uniform lower bound for smooth approximants $\omega_{{\rm KE},t} \ge C^{-1} \omega_t$
  on nearby smooth fibers. This requires to establish uniform a priori bounds for
  families  of degenerate   complex Monge-Amp\`ere potentials, 
  a theme which has known important progress in the last decade (see \cite{RZ,SSY,DGG}),
  but which still requires further understanding in order to extend Theorem \ref{thmB}
  to the positively curved setting.\\
  
  Our last result answers the Main Question positively in dimension $3$, 
   in non-positive curvature, unconditionally to any smoothability assumptions on the singularities. 
  
    \begin{bigthm}  \label{thmC}
  Let $X$ be a normal compact Kähler space of dimension three with log terminal singularities such that $K_X\sim_{\Q}\mathcal O_X$, and let $\alpha$ be a Kähler class.  The unique singular Ricci-flat metric $\omega_{\rm KE} \in \alpha$ dominates a K\"ahler form. 
 \end{bigthm}
 
 Here again, there is a version of this statement in negative curvature, cf Theorem~\ref{negative}, which surprizingly requires significantly more work than the Calabi-Yau case, cf below. The general idea behind the proof of Theorem~\ref{thmC} is that one can reduce the situation to the smooth case via a {\it two step} degeneration. More precisely it goes as follows. \\

 {\it Step 1.} One can reduce to the case where $X$ has canonical singularities, up to passing to the index one cover. 
 
 \smallskip
 
 {\it Step 2.} The first degeneration amounts to considering a terminalization $\pi: \wh X\to X$. Since $X$ is canonical, $\pi$ is crepant, i.e.   $K_{\wh X}=\pi^*K_X$ is trivial.  Now, one can realize $\pi^*\om_{\rm KE}$ as limit of singular Ricci flat metrics $\omega_\ep \in \pi^*\alpha+\ep \beta$ where $\beta$ is a Kähler class on $\wh X$. 
 
 \smallskip
 
 {\it Step 3.} Now, $\wh X$ has terminal singularities of index one, hence these are (locally) smoothable by a classification result of Reid and we can apply Theorem~\ref{thmB} to $(\wh X, \om_\ep)$ and use Chern-Lu inequality on the singular space $\wh X$ to conclude. \\
 
 In the case where $K_X$ is ample, the terminalization map is not crepant anymore. Instead, a boundary divisor $\wh \Delta$ arises on $\wh X$ so that $K_{\wh X}+\wh \Delta=\pi^*K_X$ and we are then required to generalize Theorem~\ref{thmB} to the case of klt pairs. This is not as innocuous as it may sound since even though $\wh X$ has only isolated singularities, it is not the case for the pair $(\wh X, \wh \Delta)$ anymore! Taking care of this difficulty involves an additional (third) degeneration process (cf. Theorem~\ref{klt pairs}), highlighting the guiding principle of this article.

 \subsubsection*{Contents}
 
 We recall basic facts from  analysis on complex spaces in Section \ref{sec:prelim}.
 We prove Theorem \ref{thmA} in Section \ref{sec:global},
 using uniform a priori estimates from \cite{SSY,DGG}.
We study holomorphic families of Dirichlet problems for the complex Monge-Amp\`ere equation
  in Section \ref{sec:isolated}, using  a priori estimates from \cite{GL10,GGZ20}, 
  and  prove Theorem \ref{thmB}  in Section \ref{sec:corisolated}.
  We use the previous techniques, together with some classical facts from the MMP in dimension 3,
  to establish Theorem \ref{thmC} in Section \ref{threefolds}.
  We gather in the Appendix 
  family versions of known
  estimates that are needed to establish our main results.

 \begin{ackn} 
The authors are partially supported by the research project Hermetic (ANR-11-LABX-0040),
and the ANR projects Paraplui and Karmapolis.
\end{ackn}

 \section{The Monge-Amp\`ere operator on complex spaces} \label{sec:prelim}

  In this section we let $X$ be a reduced complex analytic space  of pure dimension $n \ge 1$. 
 We will denote by $X_{\reg}$ the complex manifold of regular points of $X$. The set  
 $$
 X_{\sing} := X \setminus X_{\reg}
 $$
  of singular points is an analytic subset of $X$  of complex codimension $\ge 1$.

  \subsection{Plurisubharmonic functions}

 By definition for each point $x_0 \in X$ there exists a neighborhood $U$ of $x_0$ and a local embedding $j: U \hookrightarrow \mathbb C^N$ onto an analytic subset of  $\mathbb C^N$ for some $N \ge 1$.
 
 Using these local embeddings, it is possible to define the spaces of smooth forms of given degree on $X$.
  The notion of currents on $X$ is then defined by duality by their action on compactly supported smooth forms on $X$. The operators $\partial $ and $\bar{\partial }$, $d$, $d^c$ and $dd^c$ are then well defined by duality (see \cite{Dem85} for a careful treatment). 
  
  In the same way one can define the analytic notions of holomorphic and plurisubharmonic functions. 
  There are essentially two different notions :
 
 \begin{defi} 
 Let  $u : X \longrightarrow \mathbb R \cup \{-\infty\}$ be a given function.
 
 1. We say that   $u$ is  plurisubharmonic on $X$ if it is locally the restriction of a plurisubharmonic function on a local embedding of 
 $X$ onto an analytic subset of $\mathbb C^N$. 
 
 2. We say that  $u$ is  weakly plurisubharmonic on $X$ if $u$ is locally bounded from above on $X$ and  its restriction to the complex manifold $X_{\reg}$ is plurisubharmonic.
 \end{defi}

Fornaess and Narasimhan proved in \cite{FN} that   $u$ is plurisubarmonic on $X$ if and only if for any analytic disc 
$h : \mathbb D \longrightarrow X$, the restriction $u \circ h$ is subharmonic  on $\mathbb  D$ or identically $- \infty$.

If $u$ is  weakly plurisubharmonic on $X$,  $u$ is plurisubharmonic on $X_{\reg}$, hence upper semi-continuous on $X_{\reg}$. 
Since no assumption is made on $u$ at singular points, it is natural to extend $u$  to $X$ by the following formula :
\begin{equation} \label{eq:extension}
u^* (x) := \limsup_{X_{\reg} \ni y \to x } \,  u (y), \, \, x \in X.
\end{equation}
 The function $u^*$  is upper semi-continuous, locally integrable on $X$ and  satisfies $dd^c u^* \ge 0$ in the sense of currents on $X$. 
By Demailly \cite{Dem85},  the two notions are equivalent when $X$ is locally irreducible. More precisely we will need the following result :

\begin{thm} \cite{Dem85}
Assume that $X$ is a locally irreducible analytic space and $u : X \longrightarrow \mathbb R \cup \{-\infty\}$
is a weakly plurisubharmonic  function on $X$, then  the function $u^*$ defined  by   (\ref{eq:extension})
is plurisubharmonic on $X$.
\end{thm}

Observe that since $u$ is plurisubharmonic on $X_{\reg}$, we have $u^* = u$ on $X_{\reg}$. 
Hence $u^*$ is the upper semi-continuous extension of $u|_{X_{\reg}}$  to $X$.



\smallskip

 Following \cite{FN}  we say that $X$ is Stein if it admits a ${\mathcal C}^2$-smooth strongly plurisubharmonic exhaustion.
We will use the following definition :

\begin{defi} 
A domain $\Omega \Subset X$ is   strongly pseudoconvex if it admits a negative ${\mathcal C}^2$-smooth strongly plurisubharmonic exhaustion,
 i.e. a function $\rho$ strongly plurisubharmonic in a neighborhood $\Omega'$ of $\overline{\Omega}$ such that  $\Omega := \{ x \in \Omega' \, ; \, \rho (x) < 0\}$ and  
 $$
 \Omega_c := \{x \in \Omega'; \, \rho (x) < c\} \Subset \Omega
 $$
 is relatively compact for any $c < 0$.
  \end{defi}
 
 Our complex spaces will be assumed to be reduced, locally irreducible of dimension $n \ge 1$. 
 We denote by $\mathrm{PSH}(X)$ the set of plurisubharmonic functions on $X$.

 \subsection{Dirichlet problem on singular complex spaces} \label{sec:Dircont}
 
  The complex Monge-Amp\`ere measure $(dd^c u)^n$ of a smooth psh function in a domain of $\C^n$  is 
  the   Radon measure
  $$
  (dd^c u)^n=c \det \left( \frac{\partial^2 u}{\partial z_i \partial \overline{z}_j} \right) dV_{\rm eucl},
  $$
  where $c>0$ is a normalizing constant. The definition has been extended to any bounded psh function
  by Bedford-Taylor, who laid down the foundations of {\it pluripotential theory} in \cite{BT76,BT82}.

  The Dirichlet problem for the complex Monge-Amp\`ere operator has been studied extensively
  by many authors   (see \cite{GZbook} and the references therein).

\smallskip

 The complex Monge-Amp\`ere operator has been defined and studied 
 on complex spaces by Bedford in \cite{Bed82} and  Demailly in \cite{Dem85}. 
  It turns out that if $u \in \mathrm{PSH}(X) \cap L^{\infty}_{\rm loc} (X)$, 
  the  Monge-Amp\`ere measure $(dd^c u)^n$ is well defined on $X_{\rm reg}$
   and can be extended  to $X$ as a Borel  measure with zero mass on $X_{\rm sing}$.
  Thus all standard properties of the complex Monge-Amp\`ere operator  acting on $ \mathrm{PSH}(X) \cap L^{\infty}_{\rm loc} (X)$ 
 extend to this setting (see \cite{Bed82,Dem85}).  
 
 \smallskip
 
 The Dirichlet problem has been studied only recently in that context. We recall the following
 which is a combination of \cite{GGZ20} and \cite{DFS21,Fu21}.

\begin{thm} \label{thm:ggz&co}
Let $X$ be a Stein space of dimension $n \ge 1$, reduced and locally irreducible,
with an isolated log terminal singularity $X_{\rm sing}=\{p\}$.
Let $\Omega \subset X$ be a bounded strongly pseudoconvex domain with smooth boundary.
Fix a smooth volume form $dV$ on $X$, $\phi \in {\mathcal C}^{\infty}(\partial \Omega)$
 and $0 <f \in {\mathcal C}^{\infty}(\overline{\Omega} \setminus \{p\})$
with $f \in L^p(\Omega,dV)$ for some $p>1$. Fix $\lambda \in \R^+$.
Then there exists a unique plurisubharmonic function $u$ in $\Omega$
wich satisfies the following:
\begin{itemize}
\item[$\bullet$] $u$ is continuous on $\overline{\Omega}$ with $u_{|\partial \Omega}=\phi$;
\item[$\bullet$] $u$ is smooth in ${\Omega} \setminus \{p\}$
\item[$\bullet$] $u$ satisfies $(dd^c u)^n=e^{\lambda u}f dV$.
\end{itemize}
\end{thm}  
 
As shown further in \cite{GGZ20}, one has a uniform a priori bound on $\|u\|_{L^{\infty}(\Omega)}$
which depends on $n$, $p>1$ and $\|f\|_{L^p(\Omega)}$.
This bound only weakly depends on the geometry of $\Omega \hookrightarrow X$,
as we indicate in Theorem \ref{thm:uniformfamily}
so as to establish  a uniform family version of this estimate.

 \subsection{Canonical measure of a $\Q$-Gorenstein germ}
 \label{sec can measure}
 
 Let $(X,x)$ be a germ of normal complex space of dimension $n$ such that $mK_X$ is Cartier for some integer $m\ge 1$. If $\sigma$ is a trivialization of $mK_X$ over $X_{\rm reg}$, then the expression
  \[i^{n^2} (\sigma \wedge \bar \sigma)^{\frac 1m}\]
  defines a positive measure $\mu_{X,\sigma}$ on $X_{\rm reg}$; we still denote by $\mu_{X,\sigma}$ its trivial extension to $X$. If $\tau$ is a trivialization of $m'K_X$, then there exists $g\in \mathcal O_X(X)^*$ such that $\mu_{X,\tau}=|g|^2\mu_{X, \sigma}$ so that the qualitative behavior of the measure on the singular germ does not depend on the choice of $\sigma$. In the following, one will just write $\mu_X$ for $\mu_{X,\sigma}$. 
  
  Let $\omega$ be a smooth hermitian form on $X$, restriction of a smooth hermitian form under an embedding $(X,x)\hookrightarrow \mathbb C^N$. We denote by $f$ the density of $\mu_X$ with respect to $\omega^n$, that is, $\mu_X=f\omega^n$. In the following, the $L^p$ spaces are considered with respect to  $\omega^n$. \\

We will see below (cf proof of Lemma~\ref{smooth germ}) that $-\log f$ is quasi-psh; in particular, $\log f \in L^1$. As a consequence, one can make the following definition

\begin{defi}
Let $X,\mu_X, \omega, f$ be as above. The Ricci curvature current of $\omega$ is defined as 
  \[\Ric \, \om :=dd^c \log f= -dd^c \log\big(\frac{\om^n}{\mu_X}\big).\]
This expression yields a $(1,1)$ current with potentials which is independent of the trivialization $\sigma$. On $X_{\rm reg}$, it coincides with the usual Ricci curvature of the Kähler metric $\omega$, but we will see in the lemma below that $\Ric \, \omega$ is not a smooth form unless $X$ is smooth.
\end{defi}

    One can extend the definition of Ricci curvature for currents that are not necessarily smooth. More precisely, let $T=dd^c \varphi$ be a positive $(1,1)$ current with potential $\varphi \in L^{\infty}(X)$. Assume that $T^n=g\om^n$ has a density $g$ satisfying $\log g\in L^1$. One defines $\Ric \, T:=\Ric \, \omega - dd^c \log g$ which yields a $(1,1)$ current with potentials depending only on $T$ (and not on $\omega$ or $\varphi$). In particular, the construction can be globalized to positive $(1,1)$ currents with local potentials on a normal complex space with $\Q$-Gorenstein singularities.\\

    Recall that $(X,x)$ has log terminal singularities if and only if $\mu_X$ has finite mass, i.e. $f\in L^1$. It is standard to see that the latter condition is actually equivalent to having $f\in L^p$ for some $p>1$, cf e.g. \cite[Lemma~6.4]{EGZ09}. 
      
  The following result shows that the measure $\mu_X$ has a singular density with respect to a smooth volume form unless $(X,x)$ is smooth, compare \cite{LiZ23}. 

 \begin{lem}
 \label{smooth germ}
 Let $(X,x)$ be a germ of normal complex space of dimension $n$ such that $mK_X$ is Cartier and let $\omega$ be a smooth K\"ahler metric. Let $f$ be the density of $\mu_X$ with respect to $\omega^n$, i.e. $\mu_X=f\omega^n$. We have
 \[f\in L^\infty \quad \Longleftrightarrow \quad (X,x) \,\, \mbox{is smooth}. \]
 Moreover, the latter condition is equivalent to the existence of $k\in \R_+$ such that 
 \[\Ric \, \omega \ge - k \omega.\]
 \end{lem}
 
 \begin{proof}
 In order to lighten the notation, we will assume that $m=1$. The proof of the general case is essentially identical. One direction of the lemma is obvious, so we need to show that boundedness of the density implies smoothness of the germ. 
 
 We can assume that we have an embedding $X\hookrightarrow \mathbb C^N$ such that for all $n$-tuple $I\subset \{1, \ldots, N\}$, the linear projections $p_I:X\to \mathbb C^n$ are finite maps. It is not difficult to check that there exists a smooth function $G$ on $X$ such that $\omega^n=e^G\omega_{\mathbb  C^N}^n|_X$. 
 
 Given any $I$, there exists a holomorphic function $\zeta_I$ on $X_{\rm reg}$ such that $dz_I|_{X_{\rm reg}} = \zeta_I \sigma$. By normality of $X$, $\zeta_I$ extends to a holomorphic function on $X$. This shows that $f=(e^G\sum_I |\zeta_I|^2)^{-1}$ up to some positive constant. In particular, $-\log f$ is quasi-psh.  
 
 Next, if $f$ is bounded on $X$, then there exists $I$ such that $\zeta_I(x)\neq 0$. In particular, $\zeta_I$ is non-vanishing on the germ $(X,x)$, and $dz_I|_{X_{\rm reg}}$ is a trivialization of $K_{X_{\rm reg}}$. In other words, the map $p_I: X\to \mathbb C^n$ is étale on $X_{\rm reg}$, hence everywhere by purity of the branch locus. This shows that $(X,x)$ is smooth.  
 
 As for the last claim, assume that $(X,x)$ is singular. Then $\psi:=\log (\sum_I |\zeta_I|^2)$ is psh with analytic singularities and satisfies $\psi(x)=-\infty$ and by definition, we have $\Ric \, \om = -dd^c(\psi+G)$. If the Ricci curvature of $\omega$ were bounded from below, there would exist $A>0$ such that $dd^c(A\|z\|^2-\psi)\ge0$. In particular, $A\|z\|^2-\psi$ would be psh on $X_{\rm reg}$, hence bounded above near $x$, which is absurd. 
 \end{proof}
 
 \begin{rem}
 The Ricci curvature of $\omega$ is, however, always bounded {\it from above}. This follows from the fact that curvature decreases when passing to holomorphic submanifolds. 
 \end{rem}

 \begin{rem}
The first statement in Lemma~\ref{smooth germ} extends immediately to the setting of log pairs $(X,D)$. Indeed, if $D\neq \emptyset$, let us write $D=\sum a_iD_i$ with $a_i\neq 0$ and work at a general point $y\in D_i$ where both $X$ and $D$ are smooth so that $D_i=(g_i=0)$ for some holomorphic function $g_i$ defined near $y$. In a neighborhood of $y$, the density of $\mu_{(X, D)}$ looks like $|g_i|^{-2a_i}$ hence it is not bounded. 
 \end{rem}
 
    \begin{cor}
   \label{domination}
  Let $(X,\omega)$ be a Stein or compact Kähler space with log terminal singularities admitting a Kähler-Einstein metric $\omega_{\rm KE}$ in the sense of \cite{GGZ20} or  \cite{EGZ09} respectively. Assume that there exists $C>0$ such that 
  \[\om_{\rm KE} \le C\omega \quad \mbox{on } X. \]
  Then $X$ is smooth. 
   \end{cor}
  
 \begin{proof}
 Pick $x\in X$ and choose a neighborhood $U$ of $x$ bearing a trivialization $\sigma$ of $mK_{U_{\rm reg}}$ and write $\mu_U:=i^{n^2} (\sigma \wedge \bar \sigma)^{\frac 1m}$. Next, one can ensure that there exists $\varphi \in \mathrm{PSH}(U)\cap L^{\infty}(U)$ such that  $\omega_{\rm KE}|_{U}=dd^c \varphi$. There exist $\lambda \in \R$ and a pluriharmonic (hence smooth) function $h$ on $U$ such that 
 \[\omega_{\rm KE}^n=e^{\lambda \varphi+h} \mu_U\quad \mbox{  on }  U.\] 
 In particular, the domination $\om_{\rm KE} \le C\omega$ implies that the density of $\mu_U$ with respect to $\omega^n$ is bounded. The conclusion now follows from Lemma~\ref{smooth germ}.
 \end{proof}
  
   \section{Global smoothing} \label{sec:global}
  
  In this section we show that singular K\"ahler-Einstein metrics are K\"ahler currents
  when the variety $X$ admits a global smoothing.

 \subsection{K\"ahler-Einstein currents}

 Let $X$ be a K\"ahler normal compact space. 
  The study of complex Monge-Amp\`ere equations in this context has been initiated
 in \cite{EGZ09}, providing a way of
 constructing singular K\"ahler-Einstein metrics and extending Yau's fundamental solution to the Calabi conjecture \cite{Yau78}. More precisely, it is proven there that given a Kähler metric $\omega$ on $X$, a non-negative number $\lambda \in \{0,1\}$ and a non-negative function $f\in L^p(X)$ for some $p>1$ (satisfying $\int_X f\om^n = \int_X \om^n$ if $\lambda=0$), then the equation
 \begin{equation}
 \label{MA sing}
 (\omega+dd^c \varphi)^n=fe^{\lambda \varphi} \cdot \omega^n
 \end{equation}
 has a unique solution $\varphi \in \mathrm{PSH}(X,\omega) \cap L^{\infty}(X)$ (with the additional normalization $\sup_X \varphi =0$ if $\lambda=0$).\\
 
 Let us now explain the relation between the equation \eqref{MA sing} above and the existence of singular Kähler-Einstein metrics.
 
  We choose a pair $(X,D)$ consisting of an $n$-dimensional compact K\"ahler variety $X$ and a divisor $D=\sum a_i D_i$ with $a_i \in [0,1] \cap \mathbb Q$. We assume that there exists an integer $m\ge 1$ such that $m(K_X+D)$ is a line bundle. More precisely, we mean by this that the reflexive hull of the coherent sheaf $\left(\det(\Omega_X^1)\otimes \mathcal O_X(D)\right)^{\otimes m}$ is locally free.

 Given a hermitian metric $h$ on $K_X+D$ and the singular metric $e^{-\phi_D}$ on $X_{\rm reg}$ 
 (unique up to a positive multiple), one can construct a measure $\mu_{(X,D),h}$ on $X$ as follows. If $U$ is any open set where $m(K_X+D)$ admits a trivialization $\sigma$ on $U_{\rm reg}$, then the expression $$\frac{(\sigma \wedge \bar \sigma)^{\frac 1m}}{|\sigma|_{h^{\otimes m}}^{\frac 2m}}e^{-\phi_D}$$
defines a measure on $U_{\rm reg}$ which is independent of $m$ as well as the choice of $\sigma$ and can thus be patched to a measure on $X_{\rm reg}$. Its extension by $0$ on $X_{\rm sing}$ is by definition $\mu_{(X,D),h}$. We recall the following properties satisfied by the measure $\mu:=\mu_{(X,D),h}$, cf \cite[Lemma~6.4]{EGZ09}.
\begin{itemize}
\item The Ricci curvature of $\mu$ on $X_{\rm reg}$ is equal to $-i\Theta(h)+[D]$.
\item The mass $\int_X d\mu$ is finite if and only if $(X,D)$ has klt singularities. 
\item If $\mu$ has finite mass, then the density $f$ of $\mu$ wrt $\om^n$ (i.e. $\mu=f\cdot \om^n$) satisfies $f\in L^p(X)$ for some $p>1$. 
\end{itemize}

\medskip
From now on, we work in the following 

\begin{set}
\label{set pairs}
Let $(X,D)$ be a pair where $X$ is a compact normal Kähler space and $D$ is an effective $\mathbb Q$-divisor. Assume that $(X,D)$ has klt singularities, pick a Kähler metric $\omega$ and a hermitian metric $h$ on $K_X+D$, normalized so that $\int_X d\mu_{(X,D),h}=\int_X\om^n$. We assume either
\begin{itemize}
\item[$\bullet$]  $K_X+D$ is ample  and $\omega = i\Theta(h)$; \, or
\item[$\bullet$] $K_X+D \equiv 0$ and $h$ satisfies $i\Theta(h)=0$; \, or else
\item[$\bullet$]  $K_X+D$ is anti-ample  and $\omega = -i\Theta(h)$.
\end{itemize}
\end{set}

\begin{defi}
\label{def KE}
 In the Setting~\ref{set pairs} above, a Kähler-Einstein metric  $\om_{\rm KE}:=\omega+dd^c \varphi_{\rm KE}$ is a solution of the Monge-Ampère equation
\begin{equation}
\label{KE sing}
 (\omega+dd^c \f_{\rm KE})^n =e^{\lambda \f_{\rm KE}} \mu_{(X,D),h}
 \end{equation}
where $\lambda = 1,0$ or $-1$ according to whether we are in the first, second of third case.
 It  satisfies
 \begin{equation}
 \label{KE sing 2}
 {\rm Ric}(\omega_{\rm KE})=-\lambda \omega_{\rm KE}+[D]
 \end{equation}  in the weak sense.
 \end{defi}
 
By the results \cite{EGZ09} recalled above, \eqref{KE sing} admits a unique solution $\om_{\rm KE}$ whenever $\lambda \in \{0,1\}$. Its potential $\varphi_{\rm KE}$ is globally bounded on $X$ and $\om_{\rm KE}$ is a honest Kähler-Einstein metric on $X_{\rm reg} \setminus \mathrm{Supp}(D)$ and it has cone singularities along $D$ generically \cite{G2}. 
For $\lambda=-1$ we refer the reader to \cite{BBEGZ,Bouck18}.\\

 Kähler-Einstein theory in positive curvature is notoriously more complicated than in non-positive curvature. For that reason, we will make additional assumptions when working in the log Fano case and introduce the following 
 
 \begin{set}
 \label{set KE}
 In the Setup~\ref{set pairs} and in the case where $-(K_X+D)$ is ample, we assume additionally that
 \begin{enumerate}
 \item[$\bullet$] There exists a Kähler-Einstein metric $\om_{\rm KE}$,
 \item[$\bullet$] $\mathrm{Aut}^\circ(X,D)=0$. 
 \end{enumerate} 
  \end{set} 
The assumption on the automorphism group is to ensure uniqueness of the KE metric. Actually, assuming existence of a KE metric, its uniqueness is equivalent to the discreteness of $\mathrm{Aut}(X,D)$. 

  In summary, in Setup~\ref{set KE} above, there exists in each three cases a unique Kähler-Einstein metric $\om_{\rm KE}\in [\om]$. This is by \cite{EGZ09} if $K_X+D \ge0$ and by \cite[Theorem~5.1]{BBEGZ} if $K_X+D<0$.\\

 \subsection{K\"ahler currents}
 Let us start by recalling the following terminology. 
 
 \begin{defi}
 Let $(X,\omega)$ be a compact Kähler space, and let $T$ be a closed, positive $(1,1)$-current. We say that $T$ is a {\it Kähler current} if there exists $\ep>0$ such that the inequality $T\ge \ep \omega$ holds globally on $X$ in the sense of currents.  
 \end{defi}
 
 In our main case of interest, $T$ will have local potentials, i.e. there exist a finite open covering $X=\cup_{\alpha \in I} U_\alpha$ and functions $u_\alpha \in \mathrm{PSH}(U_\alpha)$ such that $T|_{U_\alpha}=dd^c u_\alpha$. Up to refining the cover, one can assume that $\omega|_{U_\alpha}=dd^c \varphi_\alpha$ for some strictly psh functions $\varphi_\alpha$ on $U_\alpha$. Then $T$ is a Kähler current if and only if there exists $\ep>0$ such that for all $\alpha \in I$, one has $u_\alpha-\ep \varphi_\alpha \in \mathrm{PSH}(U_\alpha)$. \\
 
 In the following, we will repeatedly use the following classical result: 
 
 \begin{lem}
 \label{extension}
 Let $X$ be a normal complex space and let $T$ be a closed $(1,1)$- current on $X$ admitting locally bounded potentials. Assume that there exists a closed analytic set $Z\subsetneq X$ such that $T|_{X\setminus Z}\ge 0$. Then, $T\ge 0$ everywhere on $X$. 
 \end{lem}
 
 \begin{proof}
 The claim is an immediate application of Chern-Levine-Nirenberg inequality that ensures that $T$ puts no mass on pluripolar sets. Alternatively, fix an open neighborhood $U$ of a point $x\in Z$ such that $T|_U=dd^c\varphi$. One can choose $\varphi$ such that $\varphi|_{U\setminus Z}$ is psh and locally bounded near $Z$, hence it extends to a psh function $\psi$ on $U$. Since $\varphi=\psi$ almost everywhere on $U$, we have $T|_U=dd^c\psi \ge 0$.
 \end{proof}
 
 The purpose of this   section is to study when $\omega_{\rm KE}$ is a K\"ahler current, by using a smoothability assumption. Our results are inspired by a result due to Ruan and Zhang \cite[Lemma~5.2]{RZ0}, and the main tool throughout the paper will be Chern-Lu inequality, which we recall below as it can be found in e.g. \cite[Proposition~7.2]{Rub14} 
 
 \begin{prop}[Chern-Lu inequality]
 \label{Chern Lu} Let $X$ be a complex manifold endowed with two Kähler metrics $\om, \wh \om$. Assume that there are constants $C_1, C_2, C_3 \in \mathbb R$ such that 
 \[\Ric \, \wh \om \ge -C_1 \wh \om-C_2 \om, \quad \mbox{and} \quad \mathrm{Bisec}_\omega \le C_3.\] 
 Then, we have the following inequality on $X$
 \[\Delta_{\wh \om} \log \mathrm{tr}_{\wh \om}\om \ge -C_1- (C_2+2C_3) \, \mathrm{tr}_{\wh \om}\om.\]

 \end{prop}

 \noindent
 Next, we introduce the following definition


\begin{defi}
Let $(X,D)$ be a klt pair with $X$ compact and let $\omega$ be a Kähler metric. We say that $(X,D,[\omega])$ admits a $\mathbb Q$-Gorenstein smoothing if there exists a triplet $(\mathcal X, \mathcal D, [\omega_{\mathcal X}])$ consisting of a normal complex space $\mathcal X$, an effective $\mathbb Q$-divisor $\mathcal D$,  and a smooth, $(1,1)$-form $\omega_{\mathcal X}$ on $\mathcal X$ admitting a proper, surjective holomorphic map $\pi:\mathcal X \to \mathbb D$ satisfying:
\begin{enumerate}
\item $K_{\mathcal X/\mathbb D}+\mathcal D$ is a $\mathbb Q$-line bundle. 
\item Every irreducible component of $\mathcal D$ surjects onto $\mathbb D$. 
\item $(\mathcal X, \mathcal D)|_{\pi^{-1}(0)} \simeq (X,D)$. 
\item For $t\neq 0$, $(X_t, D_t)=(\mathcal X, \mathcal D)|_{\pi^{-1}(t)}$ is log smooth.
\item For $t\in \mathbb D$,  $\omega_t:=\omega_{\mathcal X}|_{X_t}$ is Kähler and $[\omega_0]=[\omega]$. 
\end{enumerate}
\end{defi}

 Let us make few remarks: 
 \begin{enumerate}[label=$\circ$]
 \item If $K_X+D$ is ample or anti-ample, then so is $K_{X_t}+D_t$ for $t$ small. Then, the last condition in the definition above is automatic if $\omega \in \pm c_1(K_X+D)$, by using an embedding of $\mathcal X$ into $\mathbb P^N\times \mathbb D$ via sections of $\pm m(K_{X_t}+D_t)$ for $m$ large. 
 \item If $K_X+D$ is numerically trivial, then so is $K_{X_t}+D_t$ for  $t$ small, cf  e.g. \cite[Lemma~2.12]{DG}. 
 \item The pair $(X_t, D_t)$ is automatically klt for any $t\in \mathbb D$.
 \end{enumerate}

\begin{thm}
\label{thm strict}
Let $(X,D,[\omega])$ as in Setup~\ref{set KE} and assume that $(X,D,[\omega])$ admits a $\mathbb Q$-Gorenstein smoothing. Then, the Kähler-Einstein metric $\omega_{\rm KE}$ is a Kähler current. That is, there exists $C>0$ such that 
$$\om_{\rm KE} \ge C^{-1} \om.$$
\end{thm}

\begin{rem}
Along the same lines, one can obtain the result above assuming instead that $(X,D)$ admits a crepant resolution, i.e. a proper bimeromorphic map $p:\widetilde X\to X$ such that $K_{\widetilde X}+\widetilde D=p^*(K_X+D)$ where $\widetilde D$ is the proper transform of $D$.
\end{rem}

\begin{proof}
We consider the smoothing $\pi:(\mathcal X,\mathcal D)\to \mathbb D$. Up to shrinking $\mathbb D$ slightly and adding $\pi^*dd^c |t|^2$ to $\omega_{\mathcal X}$, one can assume that the later form is strictly positive. Since $\omega_{\mathcal X}$ is the restriction of a smooth positive $(1,1)$-form under local embedding $\mathcal X \underset{\rm loc}{\hookrightarrow} \mathbb C^N$, one can assume that the bisectional curvature of $\omega_{\mathcal X}|_{\mathcal X^{\rm reg}}$ is bounded above by a given constant $C_1$, maybe up to shrinking $\mathbb D$ just a little more. This is because the bisectional curvature decreases when passing to holomorphic submanifolds. By the same argument, we see the bisectional curvature of $\omega_t=\omega_{\mathcal X}|_{X_t}$ is bounded above by $C_1$, for any $t\in \mathbb D$. \\

For $t\in \mathbb D$, we consider $\omt=\omega_t+dd^c \varphi_t$ the unique Kähler-Einstein metric of the pair $(X_t, D_t)$. For $t\neq 0$, existence and uniqueness of $\omt$ is due to \cite{Kol98} when 
$\lambda  \ge 0$ and to \cite{SSY,LWX19}, cf also \cite{PT23} when $\lambda <0$. In the latter case, we need $t$ to be small enough. The Kähler-Einstein metric solves
\begin{equation}
\label{MAt}
(\om_t+dd^c\varphi_t)^n=e^{\lambda \varphi_t} f_t \om_t^n
\end{equation}
for some non-negative function $f_t$ on $X_t$ satisfying $\int_{X_t} f_t^p \om_t^n\le C(p,\om)$ for some $p>1$ and some constant $C(p,\om)$ independent of $t\in \mathbb D$. This is proved in \cite[Lemma~4.4]{DGG} assuming additionally $X_0$ has canonical singularities, but the klt case can be proved with minimal changes, cf also \eqref{Lp} in the next section.

If $\lambda=0$, we normalize $\varphi_t$ by $\sup_{X_t} \varphi_t =0$. We claim that up to shrinking $\mathbb D$, there exists a constant $C_2>0$ such that 
\begin{equation}
\label{borne sup 0}
\|\varphi_t\|_{L^{\infty}(X_t)} \le C_2
\end{equation}
for any $t\in \mathbb D$. 

If $\lambda>0$, this is a consequence of \cite[Theorem E]{DGG}. If $\lambda=0$, this is a consequence of the proof of \cite[Theorem~6.1]{DGG}. Indeed, in {\it loc. cit.} the fibers are assumed to be canonical but this is a mostly cosmetic assumption since going from canonical to klt singularities leaves the proof resting on \cite[Lemma~4.4]{DGG} unchanged, as explained above. Finally, if $\lambda<0$, the estimate \eqref{borne sup 0} is proved in \cite[Proposition~2.23]{SSY}, \cite[Theorem 1.2 (iii)]{LWX19} when $\mathcal D$ is pluri-anticanonical and in \cite[Theorem~5.10]{PT23} in general.\\

Set $E:=\mathrm{Supp}(\mathcal D)$. We choose a smooth hermitian metric $h$ on $\mathcal O_{\mathcal X}(E)$, a section $s$ of that line bundle cutting out $E$, and we introduce the function $\psi:=\log |s|^2_{h}$. It satisfies $dd^c \psi=-i\Theta_h( E)$ on $\mathcal X \setminus E$, hence there exists a constant $C_3>0$ such that 
\begin{equation}
\label{psi}
\psi \le C_3 \quad \mbox{and} \quad dd^c \psi \ge -C_3 \om \quad \mbox{on } \mathcal X \setminus E.
\end{equation}
We introduce a number $\ep>0$ (meant to go to zero) and we assume wlog that $\ep C_3 \le 1$. We also define $v_t:=\log \mathrm{tr}_{\omt} \om_t$, which is smooth on $X_t \setminus E$ and globally bounded (for $t\neq 0$). We want to quantify that bound.

On $X_t\setminus E$, we have $\Ric \, \omt \ge -\omt$ and $\mathrm{Bisec}(X_t,\omega_t) \le C_1$. By Chern-Lu inequality, i.e. Proposition~\ref{Chern Lu}, we get
\[\Delta_{\omt} v_t \ge -1 -2C_1 e^{v_t}\]
Using the identity $\omt=\omega_t+dd^c \varphi_t$ and the inequality \eqref{psi}, we get 
\[\Delta_{\omt} (v_t-A \varphi_t+\ep \psi) \ge e^{v_t}-C_4.\]
where $A=2(C_1+1)$ and $C_4=An+1$. The maximum of the term inside the Laplacian is attained on $X_t\setminus E$, and an easy application of the maximum principle shows
\[v_t(x) \le C_5-\ep \psi(x)\]
for any $x\in X_t \setminus E$ and $\ep>0$, and where $C_5=\log C_4+2AC_2+1$. Passing to the limit when $\ep \to 0$, we find
\begin{equation}
\label{CL}
\omt \ge C_{5}^{-1} \om_t \quad \mbox{on } X_t\setminus E, \,\, \mbox{for any} \, t\in \mathbb D^*.
\end{equation} 
\medskip

Next, we choose a continuous family of smooth maps $F_t:X_{\rm reg} \to X_t$ inducing a diffeomorphism onto their image and such that $F_0=\mathrm{Id}_{X_{\rm reg}}$.  We claim that $F_t^*\omt$ converges locally smoothly to $\om_{\rm KE}$ on $X_{\rm reg}\setminus E$ when $t\to 0$. Thanks to \eqref{CL}, this would imply that 
\[\om_{\rm KE} \ge C_5^{-1} \om \quad \mbox{on }X_{\rm reg}\setminus E\]
hence everywhere on $X$ by Lemma~\ref{extension}.  \\

Set $X^\circ:=X_{\rm reg} \setminus E$. In order to show the convergence, we claim successively:
\begin{itemize}
\item The family of functions $F_t^*\varphi_t$ is precompact in the $\mathcal C^2_{\rm loc}(X^\circ)$-topology.
\item Each cluster value $\omega_{\infty}=\om+dd^c \varphi_{\infty}$ of $F_t^*\omt$ solves the Monge-Ampère equation
\[(\om+dd^c \varphi_{\infty})^n=e^{\lambda \varphi_{\infty} }f_0\om^n \quad \mbox{on } X^\circ.\]
\item Each cluster value $\varphi_{\infty}$ is globally bounded on $X^\circ$, hence its unique $\om$-psh extension to $X$ solves the equation above on $X$. 
\end{itemize}
By uniqueness of the Kähler-Einstein metric, it would then follow that $\om_{\infty}=\om_{\rm KE}$, concluding the proof of the convergence. Let us briefly justify each of the items above. 

The first point follows from the local Laplacian estimate on $X^\circ$ obtained from \eqref{CL}. Indeed, once the laplacian estimates are obtained, one can invoke Evans-Krylov and Schauder estimates, since $F_t^*\om_t$ (resp. $F_t^*J_t, F_t^*f_t$) converges locally smoothly on $X^\circ$ to $\om$ (resp $J, f_0$) when $t\to 0$. 

The second item is an immediate consequence of the first one. As for the last one, it is a consequence of \eqref{borne sup 0}.
\end{proof}

 \section{Isolated singularities} \label{sec:isolated}
 
 \subsection{Families of Monge-Ampère equations in a local setting}
 Throughout this section, we will work in the following geometric context. 

 \begin{set}
 \label{setup}
  Let $\cX \Subset \mathbb C^N$ be a bounded normal Stein space 
  of dimension $n+1$
  endowed with a surjective holomorphic map $\pi : \cX\to \mathbb D$ such that
  \begin{enumerate}[label=$\circ$]
  \item $\cX$ is $\Q$-Gorenstein and $K_{\cX/\mathbb D} \sim_{\mathbb Q} \mathcal O_{\cX}$.
  \item For every $t\in \mathbb D$,  the schematic fiber $X_t=\pi^{-1}(t)$ is irreducible and reduced.
  \item $X_0$ has klt singularities. 
  \end{enumerate}
  In the following, we fix a basepoint $0\in X_0$. 
 \end{set}

 The first item means that there exists an integer $m\ge 1$ such that the reflexive hull of $K_{\cX}^{\otimes m}$ is trivial. Once and for all, we pick a trivialization $\Omega \in H^0(\cX, mK_{\cX})$, that is, $\Omega|_{\cX_{\rm reg}}$ is non-vanishing. 
 We set $\Omega_t:=\frac{\Omega}{(d\pi)^{\otimes m}}\big|_{X_t}\in H^0(X_t, mK_{X_t})$; this induces a trivialization of $mK_{X_t}$ and we define the measure 
 \[\mu_{X_t}:=\mu_{X_t, \Omega_t}=i^{n^2}(\Omega_t\wedge \overline \Omega_t)^{\frac 1m} \quad \mbox{on} \, X_t,\]
 cf \textsection~\ref{sec can measure}. Recall that $X_t$ has klt singularities if and only if $\mu_{X_t}$ has finite mass on each compact subset of $X_t$. Since $X_0$ is klt, inversion of adjunction \cite[Theorem~5.50]{KM} (cf also \textsection~\ref{KKMS}) shows that $\cX$ is klt in the neighborhood of $X_0$ hence for any $\cX'\Subset \cX$ there exists $\delta>0$ such that $\cX'\cap X_t$ is klt for $|t|<\delta$. We will therefore shrink $\cX$ so that each $X_t$ is relatively compact in a klt space for $|t|$ small enough.

 \begin{defi}  \label{smoothable}
We will use the following terminology. 
 \begin{enumerate}[label=$\circ$]
\item  A holomorphic map $\pi:\cX\to \mathbb D$ as in Setup~\ref{setup} is a smoothing of $X_0$ if $\pi$ is smooth over $\mathbb D^*$. 
\item A normal Stein space $X$ is smoothable if there exists a family $\pi:\cX\to \mathbb D$ as in Setup~\ref{setup} such that $X\simeq X_0=\pi^{-1}(0)$ and $\cX\to \mathbb D$ is a smoothing of $X_0$.
\end{enumerate}
 \end{defi}

 In the geometric context provided by Setup~\ref{setup}, one can consider a natural family of Monge-Ampère equations which we now describe, and whose analysis will take up most of this section.
We pick a smooth, strictly psh non-positive function $\rho$ on $\cX$ such that $\partial \cX= \{\rho=0\}$ and set $\omega= dd^c \rho$. The restriction $\om|_{X_t}$ of $\omega$ to the fiber $X_t$ will be denoted by $\omega_t$.  
 Next, we extend $F$ (resp. $h$) to a smooth function on $\overline \cX)$ (resp.  $\partial \cX$) which we still denote by $F$ (resp. $h$). 
 We also choose a (small) neighborhood $V$ of $\partial \cX$ and set $V_t=V \cap X_t$.
 We are interested in the Dirichlet problem, i.e. finding a plurisubharmonic function $u_t \in \mathrm{PSH}(X_t)\cap \mathcal C^0(\overline{X_t})$ solution of
 \begin{equation}
 \label{MA}
 \tag{MA\textsubscript{t}}
 \begin{cases}
 (dd^c u_t)^n = e^{\lambda u_t+F_t} \mu_{X_t} & \mbox{on } X_t\\
 u_t|_{\partial X_t}=h_t
 \end{cases}.
 \end{equation}
with $\lambda \in \{0,1\}$. If $X_t$ is smooth, then the existence and uniqueness of $u_t$ is classical and provides a solution 
which is smooth in $\overline{X_t}$ (see \cite{CKNS85,GL10}). In this more general setting, the existence and uniqueness of $u_t$ is provided by \cite[Theorem A]{GGZ20} for $\lambda=0$ since $X_t$ is klt up to the boundary so that the density of $\mu_{X_t}$ with respect to $\om_t^n$ belongs to $L^p(X_t, \om_t^n)$ for some $p>1$. The case  $\lambda=1$ can be treated along very similar lines. 

\smallskip

The main technical contribution of this section is summarized in the following result.

\begin{prop} \label{thm:localsmoothing}
Let $\pi: \cX\to \mathbb D$ be a family as in Setup~\ref{setup} and let $u_t$ be the solution of the Monge-Ampère equation \eqref{MA}. We have the following. 
\begin{enumerate}
\item There exist constants $ \delta, C>0$ such that for any $|t|<\delta$, 
we have $\|u_t\|_{L^{\infty}(X_t)} \le C$.
\item If $\pi$ is smooth outside the basepoint $0\in X_0$, then for any $t\neq 0$ the equation \eqref{MA} admits a smooth subsolution $\underline {u}_t$ such that 
\[dd^c \underline{u}_t \ge \ep \omega_t \,\, \mbox{on} \, X_t, \quad \mbox{and} \quad 
\|\underline {u}_t\|_{\mathcal C^k(\bar V_t)} \le C(k)\]
for any $k\in \mathbb N$, where $\ep, C=C(k)>0$ are constant that do not depend on $t$. 
\end{enumerate}
\end{prop}

The proof of Proposition~\ref{thm:localsmoothing} is quite lengthy and will be provided in \textsection~\ref{proof technical thm} below.\\
 
 Building upon Proposition~\ref{thm:localsmoothing}, one can prove the following local version of Theorem~\ref{thm strict}, which will also be useful for understanding global problems further along in the paper.

 \begin{thm}
 \label{local strict positivite}
 Let $X'\Subset \mathbb C^N$ be a normal, connected $n$-dimensional Stein space with an isolated klt singularity at the origin. Let $X\Subset X'$ be a strongly pseudoconvex domain containing the origin and such that $K_{X}\sim_{\mathbb Q}\mathcal O_{X}$. Consider the solution $u\in \mathrm{PSH}(X)\cap L^{\infty}(X)$ of 
  \begin{equation}
 \label{MA 0}
 \tag{MA}
 \begin{cases}
 (dd^c u)^n = e^{\lambda u+F} \mu_X & \mbox{on } X\\
 u|_{\partial X}=h
 \end{cases}
 \end{equation}
 where $F\in \mathcal C^{\infty}(\overline X)$, $h\in \mathcal C^{\infty}(\partial X)$ and $\lambda\in \{0,1\}$. 
 
 If $X$ is smoothable, then $dd^cu$ is a Kähler current. More precisely, there exists $C>0$ such that $dd^cu \ge C^{-1} \omega_{\mathbb C^N}|_X$. 
 \end{thm}
 
  The proof of Theorem~\ref{local strict positivite} is provided in \textsection~\ref{sec:laplacian}. Its flavor is very similar to the proof of its global counterpart, but the a priori estimates in the local setting require significantly more work. The latter are quite classical on a fixed manifold and one only needs to check that the arguments can be made to work in families, which we have chosen to do carefully in the Appendix, cf. \textsection~\ref{sec:Appendix}.

\subsection{Proof of Proposition~\ref{thm:localsmoothing}}
\label{proof technical thm}
The proof involves a careful analysis of the behavior of the measure $\mu_{X_t}$ which is achieved using the semi-stable reduction theorem following the lines of \cite[\textsection~4]{DGG}. The two items in the theorem are then proved separately. 

\subsubsection{Semi-stable model}

 \label{KKMS}
  By \cite{KKMS}, one can find a semi-stable model of $\pi$. More precisely, up to shrinking $\bD$, there exists a finite cover $\vp:t\mapsto t^k$ of the disk for some integer $k\ge 1$ and a proper, surjective birational morphism $g:=\mathcal Y \to \cX':=(\cX\times_{\varphi}\bD)^{\nu}$ from a smooth manifold $\cY$, where $\nu$ stands for the normalization, as below 
  \begin{equation}
\label{semistable}
\begin{tikzcd}
\cY  \arrow[labels=below left]{rd}{p}  \arrow[bend left]{rr}{f} \arrow{r}{g} & \cX'  \arrow{r}{h} \arrow{d}{\pi'} & \cX \arrow{d}{\pi}  \\
 &\mathbb D \arrow{r}{\vp} & \bD 
\end{tikzcd}
\end{equation}
such that around any point $y\in p^{-1}(0)$, there exists an integer $\ell\le n+1$ and a system of coordinates $(z_0, \ldots, z_n)$ centered at $y$ and such that $p(z_0, \ldots, z_n)=z_0\cdots z_\ell$. Moreover, $\pi'$ is smooth (resp. $g$ is étale) away from $h^{-1}(\mathrm{Sing}(\pi))$, and that set has $\pi'$-relative codimension at least two.

Thanks to the generic smoothness theorem, one can shrink $\bD$ so that $p$ is smooth away from $0$. In particular,  the induced morphism $f|_{Y_t}:Y_t\to X_{\varphi(t)}$ is a resolution of singularities for any $t\neq 0$, where $Y_t=p^{-1}(t)$.
We want to understand the behavior of the measures (or volume forms) $\mu_{X_t}$  when $t\to 0$. This can be achieved quite explicitly on $\cY$ via pull back by $f$ as we now explain. 

As we mentionned above, $\pi$ and $\pi'$ are smooth in codimension one (even in codimension two); this implies that $\cX'\to \cX\times_{\varphi}\bD$ is isomorphic in codimension one, hence $X_0'\to  (\cX\times_{\varphi}\bD)_0$ is finite and generically $1-1$. Next, $(\cX\times_{\varphi}\bD)_0\to X_0$ is finite and $1-1$ on points. Therefore, $X_0'\to X_0$ is finite and generically $1-1$, hence it is isomorphic since $X_0$ is normal.  In particular, $X_0'$ is irreducible and has klt singularities.

Next, we write 
\begin{equation}
\label{canformula}
K_{\cY}+Y_0=g^*(K_{\cX'}+X_0')+\sum_{i\in I} a_i  E_i
\end{equation}
where the $E_i$'s are $g$-exceptional divisors and $Y_0$ is the strict transform of $X_0'$; it is irreducible. One should observe that some of the divisors $E_i$'s may be irreducible components of $p^{-1}(0)$ so that the inclusion $Y_0 \subsetneq p^{-1}(0)$ is strict in general. The divisors $E_i$ not included in the fiber $p^{-1}(0)$ surject onto $\bD$ since we have assumed that $p^{-1}(t)$ is irreducible for $t\neq 0$. In other words, we have 
\begin{equation}
\label{intersection}
\forall i, \quad g(E_i)\cap X_0' \neq \emptyset. 
\end{equation}
The divisor $E:=\sum_{i\in I} E_i$ is the exceptional locus of $g$ and $E+Y_0$ has simple normal crossing support. Now restrict \eqref{canformula} to each irreducible component of $Y_0$ and use adjunction to obtain
\begin{equation}
\label{canformula2}
K_{Y_0}=g^*K_{X_0'}+\sum_{i\in I} a_i  {E_i}|_{Y_0}
\end{equation}
 Since $E+Y_0$ is SNC, $g|_{Y_0}: Y_0\to X_0'$ is a log resolution of $X_0'$ and we find that
 \begin{equation}
\label{klt}
\forall i, \quad a_i>-1. 
\end{equation}
 Indeed, since $X_0'$ is klt, the inequality above holds for for any $i$ such that $E_i \cap Y_0 \neq \emptyset$. And that set of indices $i$ coincides with the full set $I$   thanks to \eqref{intersection} and the connectedness of the fibers of $g$. \\

 \subsubsection{Analysis of $\mu_t$.} 
 \label{sec:uniformlocal}
 Let $\omega$ be a fixed Kähler metric on $\cX$, and let us define the function $\gamma$ on $\cX_{\reg}= \cX \setminus \mathrm{Sing}(\pi)$ by 
\[i^{n^2}(\Omega\wedge \overline \Omega)^{\frac 1m}=e^{-\gamma} \om^n\wedge d\pi \wedge \overline{ d\pi}.\]
 
The main result in this section is the following: 
\begin{lem}
\label{lem Lp}
Up to shrinking $\bD$, there exists $p>1$ and a constant $C>0$ such that for any $t\in \bD$, one has
\begin{equation}
\label{Lp}
\int_{X_t} e^{-p\gamma}\om_t^n \le C.
\end{equation}
\end{lem}

\begin{proof}[Proof of Lemma~\ref{lem Lp}]

Equation \eqref{canformula} can be reformulated by saying that the form $f^*(\frac 1 {\pi^m} \Omega)$ is a holomorphic section of $mK_{\cY}$ on $\cY\setminus p^{-1}(0)$ with a (possibly negative) vanishing order  $ma_i$ along $E_i$ and a pole or order $m$ along $Y_0$. Given $y\in Y_0$, pick a coordinate chart $(U,\underline z)$ centered at $y$ such that $p(z_0, \ldots, z_n)=z_0\cdots z_\ell$ for some $\ell \le n+1$. We can relabel the coordinates so that $(z_1, \ldots, z_n)$ are a system of coordinates on $Y_0$ and $U\cap Y_0=(z_0=0)$. Note that $(z_1, \ldots, z_n)$ remains a system of coordinates on $Y_t$ for $|t|$ small. One can choose an injection  $\{1, \ldots, \ell\}\to I$ such that $U\cap E_i=(z_i=0)$ for $1\le i\le \ell$ and $E_i\cap U=\emptyset$ for the other indices $i\in I$. 

On $U\setminus p^{-1}(0)$, one can write 
\begin{equation}
\label{pullback}
f^*\left(\frac 1 {\pi^m} \Omega\right)= a(z_0, \ldots, z_n) \left( \frac{dz_0}{z_0}\wedge dz_1 \wedge \cdots \wedge dz_n\right)^{\otimes m}
\end{equation}
for some holomorphic function $a$ on $U\setminus E$ such that $b(z):=z_1^{ma_1^-}\cdots z_\ell^{ma_\ell^-} a(z)$ extends holomorphically across $E$, where one defines $x^-:=-\min\{0,x\}$ for any $x\in \R$. Since $ma_i>-m$ from \eqref{klt}, we get $0\le ma_i^-<m$. 
 
 We have \[f^*\frac{\Omega}{(d\pi)^{\otimes m}}= \frac{f^*\Omega}{(dp)^{\otimes m}}=\frac{f^*(\frac 1 {\pi^m} \Omega)}{(\frac{dp}p)^{\otimes m}}\] as well as $\frac{dz_0}{z_0}\wedge dz_1 \wedge \cdots \wedge dz_n= \frac{dp}{p}\wedge dz_1 \wedge \cdots \wedge dz_n$ on $U$. Combining those two identities with \eqref{pullback}
 and recalling that we defined $\Omega_t=\frac{\Omega}{(d\pi)^{\otimes m}}|_{X_t}$, we get on $U\cap Y_t$:
 \[f^*\Omega_t= a(z)( dz_1 \wedge \cdots \wedge dz_n)^{\otimes m}\]
 and therefore there exists $C>0$ such that 
 \begin{equation}
 \label{mut}
 f^*\mu_{X_t} \le C \,  \frac{idz_1\wedge d\bar z_1 \wedge \cdots \wedge idz_n\wedge d\bar z_n }{\prod_{i=1}^\ell |z_i|^{2a_i^-}}. 
 \end{equation}

\medskip
Arguing as in the proof of \cite[Thm.~B.1(i)]{RZ} we can shrink $\cX$ further so that there exist bounded holomorphic functions $(\sigma_1, \cdots, \sigma_r)$ on $\cX$ satisfying $V(\sigma_1, \ldots, \sigma_r)\subset \cX_{\rm sing}$ and 
\begin{equation}
\label{gamma0}
\gamma = \log \sum_j |\sigma_j|^2+O(1)
\end{equation}

It follows from \eqref{gamma0} that there exists a constant $A>0$ such that 
$f^*\gamma \ge A \log |s_E|^2$
 where $s_E\in H^0(\cX,\cO_{\cX}(E))$ cuts out the exceptional divisor $E$ and $|\cdot |$ is a smooth hermitian metric on $\cO_{\cX}(E)$. Next, since $f:Y_t\to X_t$ is generically finite (with degree bounded independently of $t$), we get for any $p=1+\delta$: 
$$
\int_{X_t}e^{-p\gamma}\om_t^n \le \int_{Y_t}e^{-\delta f^*\gamma}f^*(i^{n^2}\Omega_t\wedge \overline \Omega_t)^{\frac 1m} 
 \le  \int_{Y_t}|s_E|^{-2\delta A} d\mu_{X_t} .
$$

Now, one can cover $Y_t$ by finitely many open sets $U_t=U\cap Y_t$ as above. On $U$, the system of coordinates $(z_0, \ldots, z_n)$ induces a system of coordinates $(z_1, \ldots, z_n)$ on $U_t$ such that we have
$$
|s_E|^{-2\delta A} \mu_{X_t} \le C \prod_{i=1}^\ell |z_i|^{-2(\delta A+a_i^-)} idz_1\wedge d\bar z_1 \wedge \cdots \wedge idz_n\wedge d\bar z_n
$$
for some uniform constant $C$ thanks to \eqref{mut}. Since $U\Subset \mathbb C^{n+1}$, the $U_t$ live in a fixed compact subset of $\mathbb C^n$ and 
the lemma follows by taking $\delta<\frac{1-\max a_i^-}{A}$. 
\end{proof}

\subsubsection{Proof of the uniform estimate.}
\label{sec:uniform estimate}
The first item of Proposition \ref{thm:localsmoothing}
is a consequence of \eqref{Lp} and
the following more general statement.

\begin{thm} \label{thm:uniformfamily}
Let $\pi: \cX\to \mathbb D$ be a family as in Setup~\ref{setup}. Fix $p>1$, $\lambda \in \R^+$, $h \in {\mathcal C}^{\infty}(\partial {\mathcal X})$
and $f_t \in L^p(X_t)$.
There exists a unique plurisubharmonic function 
$u_t \in \mathrm{PSH}(X_t)\cap \mathcal C^0(\overline{X_t})$ 
solution of
 \begin{equation*}
 \begin{cases}
 (dd^c u_t)^n = f_te^{\lambda u_t} \omega_t^n & \mbox{on } X_t\\
 u_t|_{\partial X_t}=h_t.
 \end{cases}
 \end{equation*}
 Moreover $\|u_t\|_{L^{\infty}(X_t)} \le C \|f_t\|_{L^p(X_t)}^{1/n}$, 
 where $C$ only depends on $p,n$ and $\|h_t\|_{L^{\infty}(\partial X_t)}$.
\end{thm}

\begin{proof}
The existence and uniqueness of $u_t$ is proved in \cite[Theorem A]{GGZ20} when $\lambda=0$; the case $\lambda>0$
can be treated similarly.

The key point here is to show that the solutions $u_t$ are uniformly bounded
on $X_t$ when $t\in \mathbb D$ varies, as soon as the densities $f_t$ are uniformly bounded in $L^p$.
This follows   from the analysis developed in \cite{GGZ20}, which is an extension of
Kolodziej's technique \cite{Kol98}.
Indeed  \cite[Proposition 1.8]{GGZ20} (applied with $v=0$) 
shows that $u_t$ is globally bounded on $X_t$, while
the uniform bound \eqref{Lp} allows one to show that this bound is also independent of $t$.
We provide a sketch of the proof as a courtesy to the reader.

\smallskip

\noindent {\it Step 1.}
We claim that there exists $m_1,C_1 \ge 1$ such that 
\begin{equation} \label{eq:skoda1}
\int_{X_t} \exp(-2^{-m_1}v_t) \omega_t^n \le C_1,
\end{equation}
for all $t \in \mathbb D$ and for all $v_t \in {\mathcal F}_t$, where
$$
{\mathcal F}_t:=\left\{ w \in \mathrm{PSH}(X_t) \cap L^{\infty}(X_t), \; w_{| \partial X_t}=0
\text{ and } \int_{X_t} (dd^c w)^n \le 1 \right\}.
$$

Indeed the family ${\mathcal F}_t$ is relatively compact, and any function $w$ in
$\overline{{\mathcal F}_t}$ belongs to the domain of definition of the complex Monge-Amp\`ere operator,
with zero boundary values and Monge-Amp\`ere mass less than $1$. It follows from Demailly's comparison
theorem that the Lelong number $\nu(w,x)$
is less than $1$ at a smooth point, and less than ${\rm mult}(X_t,x)^{-1/n}$ if $x$ is singular.

If $p_t:\widetilde{X_t} \rightarrow X_t$ is the blow up of a (single) smooth subvariety, then
$\nu(w,p_t(y)) \le \nu(w \circ p_t,y) \le 2 \nu(w,p_t(y))$ if $y$ 
belongs to the exceptional set $E_t$ of $p_t$,
while $\nu(w \circ p_t,y)=\nu(w,p_t(y))$ otherwise. This can be checked by embedding $X_t$ locally in $\mathbb C^N$ and using the explicit expression of the blow up of a smooth subvariety in the euclidean space. We infer that
there exists $m_1 \in \N$ such that
$\nu(w \circ f,y) \le 2^{m_1} \nu(w,f(y)) \le 2^{m_1}$
for all $y \in Y_t$.

We thus have a uniform control of the Lelong numbers of the compact 
family $f^* \overline{{\mathcal F}_t}$.
Using the subextension trick \cite[Lemma 1.7]{GGZ20}, we further reduce to controlling
$$
\int_{X'_t} \exp(-2^{-m_1}v_t) \omega_t^n =
\int_{Y'_t} \exp(-2^{-m_1}v_t \circ f) \omega_t^n
$$
on a relatively compact subset $X'_t \Subset X_t$.
We finally invoke Skoda's uniform integrability theorem
which holds for holomorphic families,
see \cite[Theorem 2.9]{DGG}.

\smallskip

\noindent {\it Step 2.}  We claim that for all compact subsets $K \subset X_t$,
\begin{equation} \label{eq:volcap}
{\rm Vol}_{\omega_t}(K) \le C_1 \exp \left( -\frac{1}{2^{m_1}{\rm Cap}(K,X_t)^{1/n}} \right),
\end{equation}
where ${\rm Cap}(K,X_t):=\sup \{ \int_K (dd^c w)^n, \, w \in \mathrm{PSH}(X_t) \text{ with } 0 \le w \le 1 \}$
denotes the Monge-Amp\`ere capacity. 

Indeed set $\lambda={\rm Cap}(K,X_t)^{-1/n}$
and $v_t=\lambda h_{K,X_t}^*$ where $h_{K,X_t}^*$ denotes the relative extremal function of the
compact set $K$ (see \cite[Definition 4.30]{GZbook}). 
It follows from \cite[Theorem 4.34]{GZbook} that $\int_{X_t} (dd^c v_t)^n=1$
and $v_t+\lambda=0$ a.e. on $K$, hence
\begin{eqnarray*}
{\rm Vol}_{\omega_t}(K)
&=& \int_K \exp \left( -2^{-m} [\lambda+v_t] \right) \omega_t^n \\
&\le &  \exp ( -2^{-m}\lambda) \int_{X_t} \exp \left( -2^{-m} v_t \right) \omega_t^n \\
&\le & C_1 \exp \left( -\frac{1}{2^{m_1}{\rm Cap}(K,X_t)^{1/n}} \right),
\end{eqnarray*}
where the last inequality follows from \eqref{eq:skoda1}.

\smallskip

\noindent {\it Step 3.} Let $\Phi$ denotes the maximal psh extension of $h$ to ${\mathcal X}$: this is the largest psh function in ${\mathcal X}$ which lies below $h$ at the boundary.
It is uniformly bounded in ${\mathcal X}$, satisfies $u_t \le \Phi_t$,
 and coincides with $h$ at the boundary.
We claim that for all $s,\delta>0$
\begin{equation} \label{eq:degiorgi}
\delta^n {\rm Cap}\{u_t-\Phi_t<-s-\delta-1\} \le 
\frac{c_{n,p} \|f_t\|_{L^p(X_t)}}{2^{m_1q}} {\rm Cap}\{u_t-\Phi_t<-s-1\}^2,
\end{equation}
where $1/p+1/q=1$. It follows indeed from \cite[Lemma 1.3]{GKZ08} that
$$
\delta^n {\rm Cap}\{u_t-\Phi_t<-s-\delta-1\} \le 
\int_{\{u_t-\Phi_t<-s-1\}} (dd^c u_t)^n.
$$
Since $(dd^c u_t)^n=f_t \omega_t^n$, we can apply H\"older inequality,
together with \eqref{eq:volcap} and the elementary inequality 
$\exp(-x^{-1/n}) \le c_{n,p} x^{2q}$, valid for all $x>0$, to conclude.

\smallskip

\noindent {\it Conclusion.}
It follows from \eqref{eq:degiorgi} that the function $g(s):={\rm Cap}\{u_t-\Phi_t<-s-1\}^{1/n}$
satisfies $\delta g(s+\delta) \le B g(s)^2$ for all $s,\delta>0$,
with $B=c_{n,p}^{1/n} \|f_t\|^{1/n}_{L^p(X_t)}2^{-m_1q/n}$.
We can invoke DeGiorgi's lemma (see \cite[Lemma 1.5]{GKZ08} with $\tau=1$) to conclude that 
$g(s)=0$ for $s \ge 4Bg(0)$.
Thus $u_t \ge \Phi_t-4Bg(0)-1$, which yields a uniform lower bound on 
$u_t$ if we can uniformly bound $g(0)$ fom above.

To estimate $g(0)={\rm Cap}\{u_t-\Phi_t<-1\}^{1/n}$ we 
let $w_t$ denote the extremal function of the set $\{u_t-\Phi_t<-1\}$. 
Recall that $-1 \le w_t \le 0$ hence
$$
w_t dd^c(\Phi_t-u_t)^n \le n(-w_t) (\Phi_t-u_t)^{n-1} dd^c u_t  \le n (\Phi_t-u_t)^{n-1} dd^c u_t.
$$
Using Stokes theorem $n$ times we thus obtain, following \cite{Blocki93}, 
\begin{eqnarray*}
{\rm Cap}\{u_t-\Phi_t<-1\} &\le&  \int_{X_t} (\Phi_t-u_t)^n (dd^c w_t)^n  \\
&\le & n! \int_{X_t} (dd^c u_t)^n \le  n! \|f_t\|_{L^p(X_t)} {\rm Vol}_{\omega_t}(X_t)^{1/q},
\end{eqnarray*}
which shows that $g(0)$ is uniformly bounded from above
by $c'_n \|f_t\|_{L^p(X_t)}^{1/n}$.
\end{proof}

\subsubsection{Existence of a suitable subsolution} 
\label{sec:subsol}


 The exceptional divisor $E$ of $g$ satisfies that there exist positive rational numbers $(b_i)_{i\in I}$ such that $-\sum_{i\in I} b_i E_i$ is $g$-ample, hence $f$-ample as well. On each $\mathcal O_{\cY}(E_i)$, one can pick a section $s_i$ cutting out $E_i$ as well as a smooth hermitian metric $h_i$ such that 
 $\rho':=f^*(A\rho)+\sum_{i\in i} b_i \log |s_i|^2_{h_i}$
  is strictly psh on $\cY$ for $A\gg 1$. 
  
From now on, we assume that $\pi$ is smooth away from our distinguished point $0\in X_0 \subset \cX$. This ensures for all $t\neq 0$, $\rho'$ is bounded on $Y_t\simeq X_{\varphi(t)}$, and that $\rho'$ is smooth on $\partial Y\simeq \partial X$.

  Next, one defines for $\delta>0$ small enough (to be determined later) the function
  \[v=v_\delta:= \rho'+\sum_{i\in I} |s_i|_{h_i}^{2\delta}.\]
Up to scaling the metrics $h_i$, we find that $v$ is strictly psh. More precisely, we can cover $\cY$ with finitely many coordinate charts $(U_\alpha)_{\alpha}$ such that
 $U_\alpha \cap E=(z_1\cdots z_{\ell_\alpha}=0)$ for some number $ \ell_\alpha \le n$ and in these charts we have
\begin{equation}
\label{minoration2}
dd^c v|_{U_\alpha} \ge c \left[ \omega_{\cY}|_{U_\alpha}+\sum_{k=1}^{\ell_\alpha} \frac{idz_k\wedge d\bar z_k}{|z_k|^{2(1-\delta)}}\right]
\end{equation}
for some $c>0$ and some fixed Kähler metric $\omega_{\cY}$ on $\cY$. In particular, it follows from \eqref{minoration2} 
above and \eqref{mut} that
\[ (dd^c v|_{Y_t})^n \ge c' f^*\mu_{X_t} \quad \mbox{on } Y_t\]
for some uniform $c'>0$. Since $v$ is uniformly bounded above on $\cY$ and $F$ is bounded on $\cX$, we can scale up $v$ so that 
 \[ (dd^c v|_{Y_t})^n \ge e^{\lambda v+f^*F_t} f^*\mu_{X_t} \quad \mbox{on } Y_t.\]
 Given $t\in \mathbb D^*$, there exists $s$ such that $\varphi(s)=t$ 
 and we denote by $v_t$ the function on $X_t$ defined by $v|_{Y_s}$ under the identification $Y_s \simeq X_t$ via $f$. Clearly, $v_t$ satisfies
  \[ (dd^c v_t)^n \ge e^{\lambda v_t+F_t} \mu_{X_t} \quad \mbox{on } X_t.\]
At this point, $v_t$ is not a subsolution of \eqref{MA} because the boundary condition is not satisfied. So we pick a cut-off function $\chi$  compactly supported on $\cX$ and satisfying $\chi \equiv 1$ near $0$. Next, we still denote by $h$ an arbitrary smooth extension of $h$ from $\partial \cX$ to $\cX$. Finally, we set 
\[\underline {u}_t:=\chi v_t+B\rho+ (1-\chi)h.\]
 One can easily see that for $B$ large enough, $\underline {u}_t$ is a subsolution of \eqref{MA}. Moreover, it is obvious on the shape of $\underline {u}_t$ that the estimates claimed in the second item of Proposition~\ref{thm:localsmoothing} are satisfied.

    \subsection{Proof of Theorem~\ref{local strict positivite}} 
  \label{sec:laplacian}
 
In this final subsection, we borrow the notation of Setup~\ref{setup} and assume that $\pi$ is smooth outside of $0\in X_0$. We consider the solution $u_t$ of \eqref{MA}.

First, we are going to derive higher order estimates near $\partial X_t$ of $u_t$ of \eqref{MA}. We will then conclude the proof of Theorem~\ref{local strict positivite} by using Chern-Lu inequality as in the proof of Theorem \ref{thm strict}.

  \begin{lem} \label{lem:gradientbord}
 There exists $C_1>0$ such that  for all $t\neq 0$,
 $\|u_t\|_{{\mathcal C}^1(\partial X_t)} \le C_1$.
 \end{lem}
 
\begin{proof}
Let $H$ be an arbitrary smooth extension of $h$ to ${\mathcal X}$.
For $A$ large enough, the function $v:=A\rho-H$ is a smooth psh function near 
$\overline{{\mathcal X}}$ such that 
$v_t=-h_t$ on $\partial X_t$.
Thus $u_t+v_t$ is psh in $X_t$ with zero boundary values. It follow from the maximum principle 
that $u_t+v_t \le 0$ in $X_t$. Using the subsolution constructed in the previous
subsection we obtain a two-sided bound
$$
\underline {u}_t \le u_t \le -v_t,
$$
with $\underline {u}_t=u_t=-v_t=h_t$ on $\partial X_t$.
The desired uniform ${\mathcal C}^1$-bound on $\partial X_t$ follows.
\end{proof}

Once $\|u_t\|_{{\mathcal C}^1(\partial X_t)}$ is under control, one 
can obtain a global control of $\|u_t\|_{{\mathcal C}^1(V_t)} $
in a neighborhood $V_t$ of $\partial X_t$ that avoids the singular point. The proof of the following proposition in given in the Appendix, cf Proposition~\ref{prop gradient estimates}.

 \begin{prop} \label{prop gradient}
 There exists $C_1'>0$ such that  for all $t\neq 0$,
 $\|u_t\|_{{\mathcal C}^1(V_t)} \le C_1'$.
 \end{prop}

In the Appendix, we explain how to derive Laplacian estimates near the boundary from lower order ones, cf. Theorem~\ref{thm laplacian estimates}. It is then straightforward to obtain the following

   \begin{prop}
   \label{prop laplacian}
   There exists $C_2>0$ such that for all $t\neq 0$, we have 
   \[\|\Delta u_t\|_{L^{\infty}(\partial X_t)} \le C_2. \]
   \end{prop}
   
   \begin{proof}
    It follows from Proposition~\ref{thm:localsmoothing} and Lemma~\ref{lem:gradientbord}
   that the assumptions of Theorem~\ref{thm laplacian estimates} are met
   and that we have uniform upper bounds on 
   $\|u_t\|_{{\mathcal C}^1(\overline{V_t})}$, $\|v_t\|_{{\mathcal C}^2(\overline{V_t})}$,
$\|f_t\|_{{\mathcal C}^1(\overline{V_t})}$, $\|h\|_{{\mathcal C}^4(\partial X_t)}$,
  $\e^{-1},\delta^{-1}$. The proposition follows. 
   \end{proof}

\begin{proof}[End of the proof of Theorem~\ref{local strict positivite}]

Pick a Kähler metric $\widehat \omega$ on $\mathcal X$ and set 
$\widehat \omega_t:=\widehat \omega|_{X_t}$. 
 Since $(dd^c u_t)^n$ is uniformly comparable to $\widehat \omega_t^n$ in a small neighborhood of $\partial X_t$,
 the  uniform bound $\|\Delta u_t\|_{L^{\infty}(\partial X_t)} \le C_2$ actually yields
a uniform constant $c_2>0$ such that 
\begin{equation} \label{eq:unifc2bis}
c_2^{-1} \widehat  \omega_t \le dd^c u_t \le c_2\widehat \omega_t
\; \; \text{ on } \partial X_t.
\end{equation}
Indeed, let $\sigma :=\inf_{t} \inf_{p\in \partial X_t} \liminf_{z\to p, z\in \partial X_t} \frac{\theta_t^n(z)}{\wh \om_t^n(z)}$; we have $\sigma>0$ since $\pi$ is smooth along $\partial \cX$ and $\|u_t\|_{L^{\infty}(X_t)}$ is uniformly bounded below by Theorem~\ref{thm:uniformfamily}. Given $p\in \partial X_t$, we have $\limsup_{z\to p, z\in X_t} \mathrm{tr}_{\wh \omega_t}\theta_t(z) \le C$ by our boundary Laplacian estimate hence 
\[\limsup_{z\to p, z\in X_t} \mathrm{tr}_{\theta_t}\wh \omega_t(z)\le \sigma^{-1} C^{n-1}\]
and \eqref{eq:unifc2bis} follows. 

Arguing as in the proof of Theorem \ref{thm strict}, 
we consider $v_t=\log {\rm tr}_{{\theta_t}}(\widehat \omega_t)$, where 
$\theta_t=dd^c u_t$ and deduce from
 the Chern-Lu formula that
 \begin{equation} \label{eq:CLb}
 \Delta_{\theta_t} (v_t-A u_t) \ge e^{v_t}-C_3,
 \end{equation} 
 for uniform constants $A,C_3>0$. We infer that 
 $v_t \le C_4$ is uniformly bounded from above. Indeed, 
either  the maximum of the function 
 $v_t-A u_t$ is reached in $X_t$ and the bound follows from
 \eqref{eq:CLb} and the uniform bound on $u_t$ or
 the maximum of $v_t-A u_t$ is reached on $\partial X_t$ and we conclude from \eqref{eq:unifc2bis}.
 
 Therefore $dd^c u_t \ge C_4^{-1}\widehat  \omega_t$ in $X_t$ for all $t \neq 0$.
Similarly to what we have done at the end of the proof of Theorem \ref{thm strict},
we conclude  by letting $t \rightarrow 0$ that
$dd^c u_0 \ge C_4^{-1} \widehat \omega_0$, 
hence $\omega_{\rm KE} =dd^c u_0$ is a K\"ahler current.
The proof of Theorem \ref{local strict positivite} is complete.
\end{proof}

 \section{K\"ahler-Einstein currents near isolated smoothable singularities}
  \label{sec:corisolated}
 
 We now use the previous analysis to establish the strict positivity of
 singular K\"ahler-Einstein metrics of non-positive curvature near smoothable isolated singularities.

  \subsection{The case of klt spaces $X$}

 \begin{thm} \label{thm:positivitylocale}
Let $X$ be a compact K\"ahler normal space with klt singularities such that either $K_X$ is ample or
$K_X \sim_{\Q} \mathcal O_X$.
Then a K\"ahler-Einstein metric $\omega_{\rm KE}$ 
 in the sense of Definition~\ref{def KE}  is a K\"ahler current near an isolated  smoothable singularity
  of $X$.
 \end{thm}

 \begin{rem}In the case where $X$ is a $\Q$-Fano Kähler-Einstein variety (i.e. $-K_X$ is ample), we expect a similar result to hold as well but this requires a better understanding of local families of K\"ahler-Einstein metrics of positive curvature.
 \end{rem}
 
 \begin{proof}
We work near  an isolated singular point $a$.
 We let $B$ denote a small strictly pseudoconvex neighborhood of 
 $a$ in $X$, isomorphic to the trace of a ball in some local embedding in $\C^N$,
 and let $\rho$ denote a local smooth potential for $\omega=dd^c \rho$ in $B$.
 
 Recall from \cite{EGZ09, Paun} that the K\"ahler-Einstein potential 
 $\f_{\rm KE}$ is smooth in $B \setminus \{a\}$.
Define $\lambda=1$ if $K_X$ is ample, and $\lambda=0$ if $K_X \equiv 0$, and
 set $F=\lambda \rho$.
The local theory recalled in Section~\ref{sec:Dircont}
shows that $\psi=\rho+\f_{\rm KE}$ is the unique solution of the Dirichlet problem
$$
(dd^c w)^n=e^{\lambda w+F} \mu_{(X,a),h} \,\, \text{ in } B,
\; \; \text{ with } \; \; 
w_{|\partial B}=\p_{|\partial B}.
$$

We assume that $(X,a)$ is a smoothable singularity in the sense of Definition~\ref{smoothable} and we let $\pi: {\mathcal X} \rightarrow \mathbb D$
denote a smoothing so that $B=\pi^{-1}(0)$ and
$X_t=\pi^{-1}(t)$ is smooth for all $t \neq 0$.
We  let $h$ denote a smooth extension of $(\rho+\f_{\rm KE})_{| \partial B}$ to 
$\partial {\mathcal X}$,
and still denote by $F$ a smooth extension of $F$ to ${\mathcal X}$. 
It follows from  Proposition \ref{thm:localsmoothing} that there exists a unique
function $u_t \in \mathrm{PSH}(X_t) \cap {\mathcal C}^0(\overline{X_t})$ such that
 \begin{equation}
 \begin{cases}
 (dd^c u_t)^n = e^{\lambda u_t+F_t} \mu_t & \mbox{on } X_t\\
 u_t|_{\partial X_t}=h_t
 \end{cases}.
 \end{equation}
together with a uniform bound $\|u_t\|_{L^{\infty}(X_t)} \le C_0$.
 We can thus apply Theorem \ref{local strict positivite} and conclude that 
 $dd^c u_0=dd^c \p=\om_{KE}$ is a K\"ahler current near $a$.
 \end{proof}

 \subsection{The case of klt pairs $(X,\Delta)$}
 \label{sec klt pair}
 We would now like to investigate the strict positivity of Kähler-Einstein currents $\omega_{\rm KE}$ associated to compact klt pairs $(X, \Delta)$ near a smoothable isolated singularity $x\in X$. There are two possibilities. \\

 {\it Case 1.} If $x\notin \mathrm{Supp}(\Delta)$, then one can find a neighborhood $U$ of $x$ such that $\partial U \cap \mathrm{Supp}(\Delta)=\emptyset$ so that $\om_{\rm KE}$ is smooth on $\partial U$, and the same arguments used in the proof of Theorem~\ref{thm:positivitylocale} will carry over mutatis mutandis to show that $\omega_{\rm KE}$ is a Kähler current near $x$. \\
 
 {\it Case 2.} If $x\in \mathrm{Supp}(\Delta)$, then as a singularity of the pair $(X,\Delta)$, it is {\it not isolated} anymore. This reflects on the metric side as well since on the boundary $\partial U$ of a small neighborhood of $x$,  $\om_{\rm KE}$ is not smooth anymore. More precisely $\omega_{\rm KE}$  has conic singularities, to be understood in a generalized sense since $\Delta$ may not have snc support near $\partial U$. Even if $\Delta$ were smooth (or snc) away from $x$, the local analysis developed so far could not be applied directly and one would have to derive boundary laplacian estimate in this singular conic setting which probably requires a lot of work. Instead we can regularize the conic singularities {\it globally} to avoid boundary problems when applying Chern-Lu inequality. This will require us to assume that each component $\Delta_i$ of $\Delta$ is $\Q$-Cartier and that {\it any} singularity of $X$ is isolated and smoothable. \\
 
 We will state the main result of this section with a slightly weaker assumption than local smoothability, which will be useful later when we work with threefolds. 
 
 \begin{defi}
 We say that a germ of normal complex space $(X,x)$ is $\Q$-smoothable if there exists a finite Galois quasi-étale cover $p:Y\to X$ with $Y$ normal and connected such that for all $y\in p^{-1}(x)$, $(Y,y)$ is smoothable in the sense of Definition~\ref{smoothable}. 
 \end{defi}
 
 In the definition above, one can always shrink $Y$ and assume that $p^{-1}(x)$ is a singleton. 
 
 It will be convenient to introduce the following setup. 
 
 \begin{set}
 \label{setup 3}
 Let $(X,\omega_X)$ be a $n$-dimensional compact Kähler space endowed with an effective $\Q$-divisor $\Delta=\sum a_i \Delta_i$ such that $(X,\Delta)$ has klt singularities. We assume that each component $\Delta_i$ of $\Delta$ is a $\Q$-Cartier divisor and that $X$ has only $\Q$-smoothable, isolated singularities.
 \end{set}
 
 We are now ready to state the main result. 
 
 \begin{thm}
 \label{klt pairs}
Let $(X,\Delta)$ be as in Setup~\ref{setup 3} and consider the unique (normalized) solution $\varphi \in \mathrm{PSH}(X,\omega_X)\cap L^{\infty}(X)$ of the Monge-Ampère equation 
\[(\om_X+dd^c \varphi)^n=e^{\lambda\varphi+F} d\mu_{(X,\Delta),h}\]
where $F\in \mathcal C^{\infty}(X)$, $h$ is a smooth hermitian metric on the $\Q$-line bundle $K_X+\Delta$ and $\lambda \in \{0,1\}$. 

Then, the current $\omega_\varphi=\omega_X+dd^c \varphi$ is a Kähler current. 
 \end{thm}
 
 As an immediate consequence of the theorem above, we get:
 \begin{cor}
 \label{cor klt pairs}
 Let $(X,\Delta)$ be as in Setup~\ref{setup 3} such that $K_X+\Delta$ is ample (resp. $K_X+\Delta \sim_{\Q} \mathcal O_X$). Then  the unique Kähler-Einstein metric $\omega_{\rm KE}$ (resp. $\omega_{\rm KE}\in \{\omega_X\}$) solving
 \[\Ric \, \om_{\rm KE}= -\omega_{\rm KE}+[\Delta] \quad \mbox{(resp. }\Ric \, \om_{\rm KE}= [\Delta] \mbox)\]
 is a Kähler current. 
 \end{cor}
 
 \begin{rem} We would like to add two comments on the result above
 \begin{enumerate}[label=$\circ$]
\item  It is conceivable that one could remove the assumption on the components $\Delta_i$ of $\Delta$ being $\Q$-Cartier by considering a $\Q$-factorialization $Y\to X$ of $X$, but it is not completely clear how the smoothability assumption would lift to $Y$. 
\item There is a noticeable difference between the assumptions of Theorem~\ref{thm:positivitylocale} and Corollary~\ref{cor klt pairs}, as in the latter one we need to assume that all singularities are isolated. It has to do with the fact that we deal with isolated singular points that may not be isolated as singularity of the pair $(X,\Delta)$ as explained in the beginning of \textsection~\ref{sec klt pair}, and this requires a subtle combination of local and global methods. If one is only interested in isolated singular points in $X\setminus \mathrm{Supp}(\Delta)$, then we would not need any global assumptions on the singularities of $X$ elsewhere, cf {\it ibid}.
\end{enumerate}
 \end{rem}
 
 \begin{proof}[Proof of Theorem~\ref{klt pairs}]
 By assumption, one can find an integer $m>0$ and sections $s_i\in H^0(X,\mathcal O_X(m\Delta_i))$ such that $\mathrm{div}(s_i)=m\Delta_i$. We pick hermitian metrics $h_i$ on the $\Q$-line bundle $O_X(\Delta_i)$ and define $|s_i|^2:=|s_i|^2_{h_i^{\otimes m}}$. Since $K_X$ is $\Q$-Cartier (as a difference $K_X=(K_X+\Delta)-\Delta)$ of $\Q$-Cartier divisors), one can find a metric $h_X$ on $K_X$ such that $h=h_X\otimes \bigotimes_i h_i^{\otimes a_i}$. Setting $b_i:=\frac {a_i} m$, one can rewrite the Monge-Ampère equation solved by $\varphi$ as 
 \[(\om_X+dd^c \varphi)^n=e^{\lambda\varphi+F} \frac{d\mu_{X,h_X}}{\prod_i |s_i|^{2b_i}}.\]
 We consider the unique (normalized) solution $\varphi_\ep$ of the regularized equation
  \[(\om_X+dd^c \varphi_\ep)^n=e^{\lambda\varphi_\ep+F} \frac{d\mu_{X,h_X}}{\prod_i (|s_i|^2+\ep^2)^{b_i}}.\]
and we set $\om_\ep:=\om_X+dd^c \varphi_\ep$. 
It follows from \cite{EGZ09,Paun}
that $\om_\ep$ is a Kähler form on $X_{\rm reg}$ for any $\ep>0$, and moreover there is a uniform constant $C_0>0$ such that 
 \begin{equation}
\label{borne Linfty}
\|\vp_\ep\|_{L^{\infty}(X)}\le C_0 \quad \mbox{and} \quad \varphi_\ep \underset{\ep \to 0}{\longrightarrow} \varphi \quad \mbox{in } L^1(X).
\end{equation}
thanks to \cite[Theorem~C]{GZ11} applied to a desingularization of $X$. Therefore, we can reduce the proof of the theorem to showing that there exists a constant $C>0$ independent of $\ep$ such that
\begin{equation}
\label{uniform ep}
\omega_\ep \ge C^{-1}\omega_X \quad \mbox{on } X.
\end{equation}

\begin{claim}
\label{non uniform}
The uniform inequality \eqref{uniform ep} holds if for any $\ep>0$, one has the qualitative inequality
\begin{equation}
\label{uniform ep 2}
\omega_\ep \ge C_\ep^{-1}\omega_X \quad \mbox{on } X,
\end{equation}
where $C_\ep>0$ is a positive constant that may depend on $\ep$.
\end{claim}

\begin{proof}[Proof of Claim~\ref{non uniform}]
Here again, the key tool is Chern-Lu inequality, Proposition~\ref{Chern Lu}. The bisectional curvature of $\om_X$ is bounded above since $\om_X$ can be extended to a Kähler metric in local embeddings in $\mathbb C^N$, and we need to bound $\Ric \, \om_\ep$ from below. 

A classical computation shows that if $s$ is a holomorphic section cutting out a divisor $D$ and if $h_D$ is a smooth hermitian metric on $L:=\mathcal O_X(D)$, then \[\Theta_{h_D}(L)+dd^c \log(|s|^2_{h_D}+\ep^2) = \frac{\ep^2}{(|s|_{h_D}^2+\ep^2)^2} \cdot |D's\wedge \overline{D's}|^2_{h_D}-\frac{\ep^2}{|s|_{h_D}^2+\ep^2} \cdot \Theta_{h_D}(L).\] In particular, if one choses a constant $C_D>0$ such that $\Theta_{h_D}(L)\le C_D\om_X$, then 
\begin{equation}
\label{minoration}
\Theta_{h_D}(L)+dd^c \log(|s|^2_{h_D}+\ep^2)\ge -C_D\om_X.
\end{equation} 
Since
\[\Ric \, \om_\ep = -\lambda \om_\ep+\lambda \om_X-dd^c F+\sum_i b_i \left( \Theta_{h_i^{\otimes m}}(\mathcal O_X(m\Delta_i))+dd^c \log(|s_i|^2+\ep^2)\right)\] 
and $F$ is smooth (so that its Hessian is bounded with respect to $\om_X$), inequality \eqref{minoration} ensures that one can find $C_1>0$ such that 
\[\Ric \, \om_\ep = -\om_\ep-C_1\omega_X.\]
Finally, let $\psi\in \mathrm{PSH}(X,\omega_X)$ be such that $(\psi=-\infty)=X_{\rm sing}$ and $\psi|_{X_{\rm reg}}\in \mathcal C^{\infty}(X_{\rm reg})$. For any $\delta>0$, the smooth quantity
\[\log \mathrm{tr}_{\om_\ep}\om_X+\delta \psi\]
on $X_{\rm reg}$ tends to $-\infty$ near $X_{\rm sing}$ thanks to our assumption \eqref{uniform ep 2}. In particular, the quantity above achieves its maximum on $X_{\rm reg}$. 
Moreover, our curvature estimates coupled with Proposition~\ref{Chern Lu} yield a constant $C_2>0$ such that  
\[\Delta_{\om_\ep} \left[\log \mathrm{tr}_{\om_\ep}\om_X+\delta \psi\right] \ge -1 -C_2\mathrm{tr}_{\om_\ep}\om_X \]
on $X_{\rm reg}$. Since $-\Delta_{\om_\ep} \varphi_\ep \ge -n+\mathrm{tr}_{\om_\ep}\om_X$, we infer that 
\[\Delta_{\om_\ep} \left[\log \mathrm{tr}_{\om_\ep}\om_X+\delta \psi-(C_2+1)\varphi_\ep \right] \ge \mathrm{tr}_{\om_\ep}\om_X-C_3 \]
where $C_3=-1+n(C_2+1)$. A classical application of the maximum principle shows that 
\[\mathrm{tr}_{\om_\ep}\om_X \le C_4 e^{\delta(\sup_X \psi- \psi)}\]
with $C_4=C_3e^{2C_0(C_2+1)}$. Passing to the limit when $\delta\to 0$, we get
\[\om_\ep \ge C_4^{-1} \om_X\quad \mbox{on } X_{\rm reg},\]
hence everywhere on $X$ by Lemma~\ref{extension}. The claim follows. 
\end{proof}
 In order to prove the theorem, we are left to establishing the qualitative estimate \eqref{uniform ep 2}. In order to achieve that, we use a local deformation argument. From now on, $\ep>0$ is fixed and all subsequent constants are allowed to depend on $\ep$. Since $\om_\ep$ is smooth on $X_{\rm reg}$, it is enough to work in small neighborhoods of the finitely many singular points in $X$. From now on, we pick a small Stein neighborhood $U'$ of $x$ admitting a quasi-étale cover $p:V'\to U'$ such that $p^{-1}(x)=\{y\}$ is a singleton and $(V',y)$ admits a smoothing $\mathcal V\to \mathbb D$ whose fibers $V_t$ satisfy $V_0\simeq V'$. 
 
 One can assume that  $U'$ is small enough so that $\om_X|_{U'}=dd^c \rho$ for some smooth strictly psh function $\rho$ on $U'$. Next, since $p$ is quasi-étale, we have $K_{V'}=p^*K_{U'}$ and the smooth hermitian metric $h_{U'}:=h_X|_{U'}$ on $K_{U'}$ pulls back to a smooth hermitian metric $h_{V'}$ on $K_{V'}$ satisfying $p^*\left(d\mu_{U',h_X}\right)=d\mu_{V',h_{V'}}$. 
 
 Next, we fix $U\Subset U'$ strongly pseudoconvex and set $V:=p^{-1}(U)$. The function $v_\ep:=p^*(\rho+\varphi_\ep)$ satisfies
\[
 \begin{cases}
 (dd^c v_\ep)^n = e^{\lambda v_\ep+p^*(F-\lambda \rho)}  \frac{d\mu_{V,h_V}}{\prod_i (|p^*s_i|^2+\ep^2)^{b_i}} & \mbox{on } V\\
 v_\ep|_{\partial V}=p^*((\rho+\varphi_\ep)|_{\partial U})
 \end{cases}
 \]
hence by Theorem~\ref{local strict positivite} the current $dd^cv_\ep$ is a Kähler current. In particular, it dominates a multiple of $p^*(\om_X)|_U$. Pushing forward, we get that in restriction to $U$, we have $\om_\ep \ge C_{\ep}^{-1} \omega_X$ and the theorem is proved.
 \end{proof}

 \section{Kähler-Einstein currents on threefolds}
 \label{threefolds}
   In this final section, we provide two results in dimension three (zero and negative curvature) ensuring that Kähler-Einstein metrics on a compact klt space are 
   Kähler currents without any extra assumption on the singularities. Although the proofs of the two results follow the same lines, we have chosen to write them separately to highlight the non-trivial simplifications occuring in the zero curvature case (or dually to insist on the consequential additional difficulties popping up in the negative curvature case). 
   
   \smallskip
   
   The crucial input specific to dimension three is Reid's classification of terminal singularities: 
   
   \begin{thm}(\cite{Reid80})
   \label{Reid}
   Let $(X,x)$ be an (isolated) terminal singularity of dimension three such that $K_X\sim_{\Z} \mathcal O_X$. Then $(X,x)$ is a compound du Val singularity. In particular, any terminal singularity of dimension three is $\Q$-smoothable. 
   \end{thm}
   
   Recall that a compound du Val singularity is a hypersurface singularity isomorphic to $(f+tg=0)\subset \mathbb C^3\times \C$ where $f,g\in \C[x,y,z]$ are such that $(f=0)\subset \C^3$ is a du Val surface singularity, cf e.g. \cite[\textsection~4.2 \& Definition~5.32]{KM}. Such a singularity can be smoothed out e.g. by $(f+tg+s=0)\subset \mathbb C^3\times \C\times \C$, where the total space is a smooth hypersurface of $\C^5$. 
   
   The second statement follows from the first after considering the quasi-étale index one cover $Y\to X$ making $K_Y$ Cartier \cite[Definition~5.19]{KM}.

  \subsection{Calabi-Yau threefolds with klt singularities}
  
  \begin{thm}
  Let $(X,\omega_X)$ be a normal compact Kähler space of dimension three with klt singularities such that $K_X\sim_{\Q} \mathcal O_X$. Then the Kähler-Einstein metric $\om_{\rm KE}\in [\omega_X]$ is a Kähler current. 
  \end{thm}
  
  \begin{proof}
  Let $p:Y\to X$ be the quasi-étale index one cover of $X$. Clearly, $p^*\om_{\rm KE}$ is the Kähler-Einstein metric in the Kähler class $p^*[\omega_X]$, so that if we prove the statement for $Y$, it will follow for $X$ by push-forward (since any Kähler form on $Y$ dominates a multiple of $p^*\omega_X$. 
  
  From now on, one can assume that $K_X\sim_{\Z} \mathcal O_X$ and, in particular, $X$ has canonical singularities. Let us consider a terminalization $\pi: \wh X \to X$ of $X$, cf \cite[Theorem~6.23]{KM}. The complex space $\wh X$ is Kähler and it has terminal singularities, hence isolated singularities. Moreover, one has $K_{\wh X}=\pi^*K_X$, i.e. $\pi$ is crepant (but in general, the exceptional locus of $\pi$ might have codimension one components). This implies that $K_{\wh X}\sim_{\Z} \mathcal O_{\wh X}$ so that $\wh X$ has smoothable singularities by Theorem~\ref{Reid}. Moreover, since $\pi$ is crepant, it is automatically isomorphic over $X_{\rm reg}$.

  We choose $\omh$ a Kähler metric on $\wh X$, and one considers the singular Ricci flat metric $\om_\ep \in [\pi^*\om_X+\ep \omh]$.  One can write $\om_\ep = \pi^*\om_X+\ep \omh+dd^c \vp_\ep$ with $\sup_{\wh X}\vp_\ep =0$, where 
  $\f_{\e}$ is a solution of 
  $(\pi^*\omega_X+\e \omh+dd^c \f_{\e})^n=c_{\e} \mu_{\hat{X}}=c_{\e} \hat{f} \omh^n$,
  where $c_\ep =\frac{[\pi^*\omega_X+\ep \omh]^n}{ \mu_{\wh X}(\wh X)}$ 
  and  $\hat{f} \in L^{1+\delta}(\wh X,\wh \om^n)$ with  $\delta>0$,   since $\hat{X}$ has terminal singularities. 
  It follows from the techniques in \cite{EGZ08,DemPal} that 
  \begin{equation}
  \label{C0 estimate}
  \|\vp_\ep\|_{L^{\infty}(\wh X)} \le C_0, \quad \mbox{and} \quad  \om_\ep \underset{\ep \to 0}{\longrightarrow} \pi^*\om_{\rm KE} \quad \mbox{weakly.}
  \end{equation} 
  where $C_0>0$ is independent of $\ep$. The discerning reader will have noticed that the situation is slightly different from the one in {\it loc. cit.} since $X$ is singular. To patch this little gap, one could e.g. appeal to \cite[Theorem~A]{DGG} applied to a desingularization of $\widetilde X\to \wh X$ with reference form the pull-back of $\pi^*\omega_X+\ep \omh$ to $\widetilde X$ and this would yield the uniform estimate. As for the stability statement, it is a classical consequence of the uniform estimate: one first gets higher order estimates locally on $X_{\rm reg}$ using Tsuji's trick and then one uses uniqueness of the Kähler-Einstein metric to conclude that the relatively compact family $(\varphi_\ep)_{\ep >0}$ has a single cluster value in $L^1(\wh X)$ which is nothing but $\varphi_0$. \\

  The main point to establish is that there is a constant $C>0$ independent of $\ep$ such that 
  \begin{equation}
  \label{main point}
  \om_\ep \ge C^{-1} \pi^*\omega_X
  \end{equation}
  Indeed, passing to the limit when $\ep \to 0$ would then imply that $\pi^*\om_{\rm KE}\ge C^{-1} \pi^*\om_X$ and the theorem would follow by pushing forward by $\pi$. 
  
  Now we know that inequality \eqref{main point} holds for any $\ep>0$ with a constant $C=C_\ep$ depending a priori from $\ep$, as this is the content of Theorem~\ref{thm:positivitylocale}. As before, we will use Chern-Lu inequality to make this qualitative control quantitative. \\

Choose a function $\psi\in \mathrm{PSH}(X,\omega_X)$ be such that $(\psi=-\infty)=X_{\rm sing}$ and $\psi|_{X_{\rm reg}}\in \mathcal C^{\infty}(X_{\rm reg})$. For any $\delta>0$, the smooth quantity
\[\log \mathrm{tr}_{\om_\ep}\pi^*\om_X+\delta \psi\]
on $\pi^{-1}(X_{\rm reg}) \subset \widehat X \setminus \mathrm{Exc}(\pi)$ tends to $-\infty$ near $\pi^{-1}(X_{\rm sing})$ since $\om_\ep$ is a Kähler form. 
In particular, the quantity above achieves its maximum on $X_{\rm reg}$.
 
Next, the bisectional curvature of $\pi^*\omega_X$ is well-defined and bounded above uniformly on $\pi^{-1}(X_{\rm reg})$, while $\Ric \,\om_\ep= 0$.  By Chern-Lu inequality
(Proposition~\ref{Chern Lu}) we find a constant $C_2>0$ such that  
\[\Delta_{\om_\ep} \left[\log \mathrm{tr}_{\om_\ep}\pi^*\om_X+\delta \psi\right] \ge  -C_2\mathrm{tr}_{\om_\ep}\pi^*\om_X \]
on $\pi^{-1}(X_{\rm reg})$. 
Since $-dd^c \varphi_\ep=-\om_\ep+\pi^*\om_X+\ep \omh$, we have 
\[-\Delta_{\om_\ep} \varphi_\ep \ge -n+\mathrm{tr}_{\om_\ep}\pi^*\om_X,\]
hence  
\[\Delta_{\om_\ep} \left[\log \mathrm{tr}_{\om_\ep}\pi^*\om_X+\delta \psi-(C_2+1)\varphi_\ep \right] \ge \mathrm{tr}_{\om_\ep} \pi^*\om_X-C_3 \]
where $C_3=-1+n(C_2+1)$. 
Using \eqref{C0 estimate} we obtain
in the end 
\[\mathrm{tr}_{\om_\ep}\pi^*\om_X \le C_4 e^{\delta(\sup_X \psi- \psi)}\]
 with $C_4=C_3e^{C_0(C_2+1)}$, holding on $\pi^{-1}(X_{\rm reg})$ for any $\delta>0$. Equivalently, we have $\om_\ep \ge C_4 e^{-\delta(\sup_X \psi- \psi)}\, \pi^*\om_X$. After passing to the limit when $\delta\to 0$, one can appeal to  Lemma~\ref{extension} to obtain \eqref{main point}, hence the theorem. 
  \end{proof}

  \subsection{Canonically polarized threefolds with klt singularities}
  
    \begin{thm}
    \label{negative}
  Let $X$ be a normal projective variety of dimension three with klt singularities such that $K_X$ is ample. Then the Kähler-Einstein metric $\om_{\rm KE}$ is a Kähler current. 
  \end{thm}
  
  \begin{proof}
  The reduction to the terminal case is a bit more involved here. More precisely, we cannot easily assume that $X$ has canonical singularities without resorting to a local argument which would collapse since the singularities of $X$ are not isolated. Instead, we use a $\Q$-factorial terminalization  $\pi:\wh X\to X$ whose existence (in any dimension) is guaranteed by \cite[Corollary~1.4.3]{BCHM}. The map $\pi$ is such that there exists an effective divisor $\wh \Delta$ on $\wh X$ such that $K_{\wh X}+\wh \Delta=\pi^*K_X$ and $(\wh X, \wh \Delta)$ is a terminal pair, hence $\wh X$ has isolated $\Q$-smoothable singularities by Theorem~\ref{Reid}. Moreover, $\wh X$ is $\Q$-factorial. 
  
  Next, let $\omh$ be a Kähler metric on $\wh X$ and let us consider the twisted Kähler-Einstein metric $\om_\ep \in c_1(\pi^*K_X)+\ep [\omh]$, i.e. the solution of 
  \[\Ric \, \om_\ep = -\om_\ep+\ep \omh+[\wh \Delta].\]
  If $\om_X\in c_1(K_X)$ is a Kähler metric, say the curvature of an hermitian metric $h$ on $K_X$, one can write $\om_\ep= \pi^*\om_X+\ep \omh+dd^c \vp_\ep$ where 
  $\f_{\e}$ is a solution of 
  \[(\pi^*\omega_X+\e \omh+dd^c \f_{\e})^n=e^{\f_{\e}} \mu_{(\hat{X},\wh \Delta), \pi^*h}=e^{\f_{\e}} \hat{f} \omh^n,\]
  with $\hat{f} \in L^{1+\delta}(\wh X,\wh \om^n)$ with $\delta>0$  since $\hat{X}$ has terminal singularities.
As in the Calabi-Yau case, one has
    \begin{equation}
  \label{C0 estimate 2}
  \|\vp_\ep\|_{L^{\infty}(\wh X)} \le C, \quad \mbox{and} \quad  \om_\ep \underset{\ep \to 0}{\longrightarrow} \pi^*\om_{\rm KE} \quad \mbox{weakly.}
  \end{equation} 
  where $C>0$ is independent of $\ep$. The only difference with the Calabi-Yau case is that we first need to show that there exists $C>0 $ independent of $\ep$ such that 
  \begin{equation}
  \label{borne sup}
  \sup_{\wh X}\vp_\ep \le C
  \end{equation}
This is certainly classical, but we recall the argument for the reader's convenience. First, one can assume without loss of generality that $\mu:=\mu_{(\hat{X},\wh \Delta), \pi^*h}$ is a probability measure. By Jensen's inequality, we infer that 
  \[\int_{\wh X} \vp_\ep d\mu \le \log \int_{\wh X}(\pi^*\omega_X+\e\wh \omega)^n \le C_{1}\]
  for some $C_{1}>0$ independent of $\ep$.  Let $p:\widetilde X\to \widehat X$ be a resolution of singularities and let $\omega_{\widetilde X}$ be a Kähler metric on $\widetilde X$. Up to scaling the latter metric, one can assume that $p^*(\pi^*\omega_X+\e \omh)\le \omega_{\widetilde X}$ and $p^*\mu\le \omega_{\widetilde X}^n$ where the latter follows since $\wh X$ has terminal (hence canonical) singularities. By the usual compactness properties of $\omega_{\widetilde X}$-psh functions, we find a constant $C_{2}$ such that 
 \[ \int_{\widetilde X}(\sup_{\widetilde X}\psi-\psi)\om_{\widetilde X}^n \le C_{2}\]
 for any $\psi\in \mathrm{PSH}(\widetilde X, \om_{\widetilde X})$. Now, 
 \begin{align*}
 \sup_{\wh X} \vp_\ep &= \int_{\wh X}(\sup_{\wh X} \vp_\ep-\vp_\ep)d\mu+\int_{\wh X} \vp_\ep d\mu\\
 & \le \int_{\wh X}(\sup_{\widetilde X} p^*\vp_\ep-p^*\vp_\ep)\om_{\widetilde X}^n+C_1\\
 & \le C_1+C_2 
 \end{align*}
 since $p^*\vp_\ep \in \mathrm{PSH}(\widetilde X, \om_{\widetilde X})$. Hence \eqref{borne sup} follows.

Coming back to proof of  the main result, the remaining point to establish is that there is a constant $C>0$ independent of $\ep$ such that 
  \begin{equation}
  \label{main point 2}
  \om_\ep \ge C^{-1} \pi^*\omega_X.
  \end{equation}
  Thanks to Theorem~\ref{klt pairs}, inequality \eqref{main point 2} holds for any $\ep>0$ with a constant $C=C_\ep$ depending a priori from $\ep$. In order to make this bound quantitative, the exact same strategy as in the Calabi-Yau case applies.   
  \end{proof}

\section{Appendix: Gradient and Laplacian estimates in families} \label{sec:Appendix}

In this Appendix we  provide the proofs of the family version of the gradient estimate (Proposition~\ref{prop gradient estimates}) and the Laplacian estimate (Theorem~\ref{thm laplacian estimates}) as a courtesy to the reader.

\subsection{Gradient estimate in families}

Let us consider the solution $u_t$ of \eqref{MA}, and recall that $f:{\mathcal Y} \rightarrow {\mathcal X}$ be a log resolution of singularities after base change. We recall that $\mathcal Y$ is smooth and that $E:=\mathrm{Exc}(f)\subset \mathcal Y$ is a divisor satisfying $E\subset Y_0$ since we assume that $\pi$ is smooth outside of $0\in X_0$.  Moreover, the map $f_t:Y_t\to X_t$ is a finite étale cover of constant degree for any $t\in \mathbb D^*$. Up to picking a connected component of $Y_t$, there is no loss of generality in assuming that the latter map is actually isomorphic. 

Recall that the function $\rho'=f^*(A\rho)+\sum_{i\in I} b_i \log|s_i|^2_{h_i}$ is strictly psh on ${\mathcal Y}$,
where $E_i=(s_i=0)$ are the irreducible components of the exceptional locus $E$ of $f$. 

Without loss of generality, one can assume that the positive numbers $b_i$ satisfy $\max_i b_i \le 1$ and that $\log |s_i|^2_{h_i} \le 0$ for any $i$.  Next, one can write 
\[dd^c \rho'=\omega+ \sum b_i[E_i]\] 
so that $\omega$ is a Kähler form on $\mathcal Y$. Setting 
\[v:=f^*(u - A\rho)-\sum_i b_i \log|s_i|^2_{h_i},\] 
we see that the following equation holds on $Y_t$ for any $t\neq 0$:
\[
\begin{cases}
(\omega_t+dd^c v_t)^n = (dd^c f^*u_t)^n= 
e^{\lambda f^*u_t +f^*F_t } f^* \mu_t=\frac{
e^{\lambda v_t+G_t}}{\prod_i |s_i|^{2(a_i-b_i)}} \omega_t^n \\ 
{v_t}|_{\partial Y_t}=(f^*h-\sum b_i \log |s_i|^2_{h_i})|_{\partial Y_t}
\end{cases}
\]
with $G$ smooth on $\cY$, and $a_i <1$ (not necessarily positive) as explained in \textsection~\ref{sec:uniformlocal}. In the following, one sets $|s|^2:=\prod_{i\in I} |s_i|^2_{h_i}$.

 \begin{prop}  
 \label{prop gradient estimates}
 In Setup~\ref{setup}, we assume that $\pi$ is smooth outside of the basepoint $0\in X_0$ and we let $u_t$ be the solution of \eqref{MA}. 
 There exist constants $C>0$ and $N>0$ such that  for all $t\neq 0$,
 \begin{equation}
\| |s|^N \nabla_{\omega_t} v_t\|_{L^{\infty}(Y_t)}\le C.
\end{equation}
for some constant $C>0$ independent of $t$.
 \end{prop}

 \begin{proof}
 This is a family version of \cite[Proposition 2.2]{DFS21},
 which itself is an extension of previous similar estimates of Blocki \cite{Blocki09}
 and Phong-Sturm \cite{PS12}.
 
 In order to increase the readability of the proof, we will soon skip the subscripts $t$ appearing in all the quantities involved in our problem, but before we do that, let us emphasize explicitly what are the uniform bounds that we will later rely on. In order to do so, one rewrites the equation solved by $v_t$ as 
 \[(\om_t+dd^c v_t)^n=e^{\lambda v_t+f_t}\om_t^n\]
 and we have uniform constants $C_i$ (independent of $t\in \mathbb D^*$) such that one has the a priori estimates
 \begin{equation}
 \label{estimates}
 \|v_t+\sum_i b_i \log |s_i|^2_{h_i}\|_{L^{\infty}(Y_t)} \le C_1, \quad  \|\nabla_{\omega_t} v_t\|_{L^{\infty}(\partial Y_t)} \le  C_2, \quad \big|\nabla_{\omega_t} f_t\big|\le \frac{C_3}{|s|}
 \end{equation}
 as well as the following inequalities 
\begin{equation}
\log |s|^2 \le C_4, \quad |\nabla_\omega |s|^2|^2 \le C_5, \quad dd^c \log |s|^2\ge -C_6\omega, \quad  |\mathrm{Bisec}_{\om_t} |\le C_7
 \end{equation}
 holding on $Y_t$ for all $t\neq 0$. The a priori estimates \eqref{estimates} have already been established, cf \textsection~\ref{sec:uniform estimate} and Lemma~\ref{lem:gradientbord}. From now on, we drop the subscript $t$. \\

 Define $\beta := \vert \nabla v \vert^2_\omega$ and $\alpha := \log \beta - \gamma (v) + k \log |s|^2$, where $k \in \mathbb N_{>0}$ is number that we will fix later (we will actually see that any number $k\ge 2$ is suitable) and $\gamma$ is of the following form
\[ \gamma (t) = A (t - \frac{1}{t + B}),\]
 where $A, B > 0$ are constants to be chosen later.  We should however observe already that the first estimate in \eqref{estimates} shows that for $B:=C_1+1$, one has
 \begin{equation}
 \label{gamma}
-A(C_1+1) \le \gamma(v) \le  A(-\log |s|^2+C_1)
 \end{equation}
 as well as 
  \begin{equation}
 \label{gamma1}
A\le \gamma'(v) \le 2A, \quad \mbox{and} \quad  - \gamma''(v) \ge \frac{2A}{(2B-\log |s|^2)^3}.
 \end{equation}

 \medskip
 
 The function $\alpha$ attains its maximum at a point $x_0 \in \overline{Y_t} $. We can assume that $x_0 \in Y_t$, otherwise we are done by \eqref{estimates}.
 We choose normal complex coordinates $(z_1, \cdots, z_n)$ centred at $x_0$ so that $\omega_{x_0} = \sum dz_j \wedge d \bar z_j$ and 
 $\theta := \omega + dd^c v = \sqrt{-1} \sum_{i,j}  \theta_{i \bar j} d z_i \wedge d \bar z_j$ is diagonal at $x_0$ i.e.
 $ \theta_{i \bar j} (x_0) = \delta_{i j} \theta_{i \bar i}.$
 
 \medskip
 
 \noindent
 By \cite[eq. (1.14)]{Blocki09} (cf. also \cite[eq. (2.7)]{PS12}) we have 
 \begin{eqnarray*}
&& \Delta_{\theta} (\log \beta - \gamma(v)) = \frac{1}{\beta}\left\{\vert \nabla \nabla v\vert^2_{\omega,\theta}  +  \vert \overline \nabla \nabla    v \vert^2_{\omega,\theta} +  2 \Re (\partial^p (\lambda v+ f) \, \partial _p v)\right\} \\
     \\
 && + \frac{1}{\beta}\left\{\theta^{k \bar l} \partial_p v {R^p_{k \bar l}}^{\bar m} \partial_{\bar m} v \right\}   - \frac{1}{\beta^2} \vert \nabla \beta\vert^2_{\theta}  - \gamma''(v) \vert \nabla v\vert^2_{\theta} - \gamma'(v) \Delta_\theta v.
 \end{eqnarray*}
 We have the following easy estimates 
 \begin{enumerate}[label=$\bullet$]
 
 \item  $ \theta^{k \bar \ell } \partial_p v {R^p_{k \bar \ell l}}^{\bar m} \partial_{\bar m} v \ge ( \inf \mathrm{Bisec}_{\omega}) \,\theta^{k \bar \ell} \omega_{k\bar \ell} \omega^{p \bar m} \partial_p v \partial_{\bar m} v  \ge - C_7 \beta \, \text{tr}_{\theta} \omega,$
  \item   $   \partial^p \lambda v \, \partial _p v =\lambda \beta,$

 \item   $  \vert \partial^p f\, \partial _p v\vert \le \vert \nabla_\omega  f\vert \sqrt{\beta} \le C_3 \sqrt{\beta} \vert s\vert^{-1},$
  
  \item   $  - \Delta_\theta v = - n + \text{tr}_{\theta} \omega. $
    \end{enumerate}
    On the other hand, since $\Delta_\theta \log \vert s\vert^2 \ge - C_6\,   \text{tr}_{\theta} \omega$, it follows that
    \[    \Delta_\theta \alpha \ge  \Delta_\theta (\log \beta - \gamma (v))  - k C_6\,   \text{tr}_{\theta} \omega.\]
    All in all, we get

   \begin{eqnarray}
   \label{preparation}
     \Delta_\theta \alpha & \ge & \left(\frac{\vert \nabla \nabla v\vert^2_{\omega,\theta}}{\beta} -\left| \frac{\nabla \beta}{\beta}\right|^2_\theta\right)- \gamma''(v) |\nabla v|^2_\theta\\
     &&-\frac{2C_3}{\sqrt \beta |s|}+(\gamma'-kC_6-C_7)\mathrm{tr}_\theta\omega-n\gamma' \nonumber
   \end{eqnarray}
   
   \medskip 
   
   \noindent
    We now deal with the first summand of the RHS of \eqref{preparation}. At the point $x_0 \in Y_t$ where $\alpha$ achieves its maximum, we have the critical equation
    \[    \frac{\nabla \beta}{\beta} - \gamma'(v) \nabla v + k \frac{\nabla \vert s\vert^2}{\vert s\vert^2} = 0.\]
    On the other hand, by  \cite[eq. (1.15)]{Blocki09} (cf. also \cite[eq. (2.12)]{PS12}) we have,
    \[    \frac{1}{\beta}\vert \nabla \nabla v\vert^2_{\omega, \theta}\ge \left\vert \frac{\nabla \beta}{\beta}  - \frac{1}{\beta} \nabla v\cdot h+ \frac{1}{\beta} \nabla v\right\vert_\theta^2,\]
    
    where $h$ is the smooth section of $\mathrm{End}(T_{Y_t})$ defined by $\theta$ via the metric $\omega$ i.e. locally $h^i_j = \omega^{ i  \bar \ell}\theta_{\bar \ell j}$.    Set 
    \[a := \gamma'(v) \nabla v - k \frac{\nabla \vert s\vert^2}{\vert s\vert^2} \underset{\mathrm{at} \, x_0}{=}  \frac{\nabla \beta}{\beta},\] 
    and 
    \[  b:= \frac{1}{\beta} (\nabla v - \nabla v \cdot h). \]
    Then at the point $x_0$ we have
    \begin{eqnarray}
   \label{PS-1}  \frac{1}{\beta} \vert \nabla \nabla v\vert^2_{\omega, \theta} \ge  \vert a+b \vert^2_{\theta} &\ge & \vert a\vert^2_{\theta} + 2 \Re \langle a, b\rangle_{\theta} \\
     & = &  \left\vert \frac{\nabla \beta}{\beta} \right\vert^2_{\theta} +  2 \Re \langle a, b\rangle_{\theta}\cdot \nonumber
    \end{eqnarray}
    Now we analyze the term $\langle a, b\rangle_{\theta}$. At the point $x_0$ we have 
    \begin{enumerate}[label=$\bullet$]
    \item 
     $ \langle \nabla v, \nabla v \cdot h\rangle_{\theta}  =  \theta^{i \bar j} \nabla_i v \nabla_{\bar \ell} v \, \omega^{\bar \ell p} \theta_{p \bar j}
       = \omega^{p \bar p} \nabla_p v \nabla_{\bar p} v = \beta$
    
    \item $\langle \nabla \vert s\vert^2, \nabla v \cdot h\rangle_{\theta} =   \theta^{i \bar j} \nabla_i \vert s\vert^2  \nabla_{\bar \ell} v \,  \omega^{\bar \ell p} \theta_{p \bar j}  = \theta^{i \bar i} \nabla_i \vert s\vert^2  \nabla_{\bar i} v\,  \theta_{i \bar i}  = \langle \nabla \vert s\vert^2, \nabla v \rangle_{\omega} $
    \end{enumerate}
  Therefore we conclude that at $x_0$, we have
  \begin{eqnarray}
\label{PS0}  \langle a, b\rangle_{\theta}& = & \frac{\gamma'(v)}{\beta} \left(\vert \nabla v\vert^2_{\theta} - \beta\right)+ \\
  &&+  \frac{k}{\beta \vert s\vert^2} \left(\langle \nabla \vert s\vert^2, \nabla v \rangle_{\theta} - \langle \nabla \vert s\vert^2, \nabla v \rangle_{\omega}\right) \nonumber
  \end{eqnarray}
    Observe that 
\[   \left\vert \nabla \vert s\vert^2 \right\vert^2_{\theta} = \theta^{i \bar i} \left\vert \nabla_i \vert s\vert^2\right\vert^2 \le  \vert \nabla_\omega \vert s\vert^2\vert^2\,  \text{tr}_{\theta} \omega  \le C_5 \, \text{tr}_{\theta} \omega,\]  
    so that
\begin{equation}
\label{PS1}
2  \left\vert \langle \nabla \vert s\vert^2, \nabla v \rangle_{\theta}\right\vert  \le C_5\, \text{tr}_{\theta} \omega + \vert \nabla v\vert_{\theta}^2.
  \end{equation}
    By Cauchy-Schwarz inequality again, we have 
   \begin{eqnarray}
  \label{PS2}   \frac{2}{\beta \vert s\vert^2} \left\vert\langle \nabla \vert s\vert^2, \nabla v \rangle_{\omega}\right\vert &\le &\left\vert \nabla \vert s\vert^2 \right\vert^2_{\omega} +  \frac{1}{(\beta \vert s\vert^2)^2}\vert \nabla v \vert^2_{\omega} \\
     & \le & C_5+ \frac{1}{\beta \vert s\vert^4}. \nonumber
    \end{eqnarray}
    Combining \eqref{PS-1}-\eqref{PS0} with \eqref{PS1}-\eqref{PS2}, we get at $x_0$

    \begin{equation}
\label{PS3} 
\frac1 \beta \vert \nabla \nabla v\vert^2_{\omega,\theta} -\left| \frac{\nabla \beta}{\beta}\right|^2_\theta  \ge - \gamma'(v) - \frac{k}{\beta \vert s\vert^2} \left( C_5 \text{tr}_{\theta} \omega + \vert \nabla v\vert_{\theta}^2\right) - k (C_5+  \frac{1}{\beta \vert s\vert^4}).
\end{equation}

\bigskip

On can now put together all our estimates. More precisely, if we combine \eqref{preparation} with \eqref{PS3}, the maximum principle yields the following inequality at $x_0$
    \begin{eqnarray*}
0 & \ge & \left(\gamma'(v) - (kC_6+C_7) - \frac{kC_5}{\beta \vert s \vert ^2}\right) \text{tr}_{\theta} \omega + \left(- \gamma''(v) - \frac{k}{\beta \vert s\vert^2}\right) \vert \nabla v\vert^2_{\theta} \\
&&-\left((n+1) \gamma'(v) + k\big(C_5+ \frac{1}{\beta \vert s\vert^4} \big)+  \frac{2C_3}{\sqrt{\beta}|s| }\right).
  \end{eqnarray*}
  
  \begin{claim}
  \label{claim gradient}
  Given some number $1\le p\le k$, one can assume that at the point $x_0$ we have $\beta (x_0) \vert s(x_0)\vert^{2 p} \ge 1$. 
  \end{claim}
  
  \begin{proof}[Proof of Claim~\ref{claim gradient}]
  Assume that $\beta (x_0) \vert s(x_0)\vert^{2 p} \le 1$, and consider the function $G:=\log\beta+N\log |s|^2=\alpha+\gamma(v)+(N-k)\log |s|^2$ for some $N$ to fix later. We want to show that for some suitable $N$, $G$ is uniformly bounded above. We have
  \begin{eqnarray*}
  G & \le & \alpha(x_0)+\gamma(v)+(N-k)\log |s|^2\\
   & \le &  -\gamma(v)(x_0)+(k-p)\log|s|^2(x_0)+\gamma(v)+(N-k)\log |s|^2\\
   & \le & A(2C_1+1)+(k-p)\log|s|^2(x_0)+(N-k-A)\log |s|^2
  \end{eqnarray*}
  thanks to \eqref{gamma} and the claim follows by taking $N=k+A$.
  \end{proof}
  
  At this point, it is important to assume that $k\ge 2$ so that one can ensure that $\beta \vert s\vert^{4}(x_0)\ge 1$ thanks to the claim. From now on, we fix $k:=2$ (but any number $k\ge 2$ would do too). In particular, we have $ \frac 1{\beta |s|^2(x_0)} \le 1$  and $\beta \vert s\vert^2 \ge \vert s\vert^{-2}$ at $x_0$. One chooses 
 \[A:=  \max\left\{k(C_5+C_6)+C_7+1, k\cdot \sup_{x\in [0,1]} x(2B-\log x)^3\right\}.\] 
 Given \eqref{gamma1}, one gets at $x_0$
\[     - \gamma''(v) - \frac{k}{\beta \vert s\vert^2} \ge  \frac{2 A}{(2B - \log \vert s\vert^2)^3} -  k  \vert s\vert^2 >   \frac{ A}{(2B - \log \vert s\vert^2)^3}. \]
    Therefore we have at $x_0$:
\[    \text{tr}_{\theta} \omega + \frac{ A}{(2B - \log \vert s\vert^2)^3}  \vert \nabla v\vert^2_{\theta} \le  C_8,\]
with $C_8:=2C_3+k(1+C_5)+2(n+1)A$. In particular, we have 
   \[     \text{tr}_{\theta} \omega (x_0) \le C_8, \quad \mbox{and} \quad  \vert \nabla v\vert^2_{\theta} (x_0) \le C_8 (2B - \log \vert s\vert^2)^3.\]
    Since $\theta^n \lesssim |s|^{-2} \omega^n$, we infer from the above estimate that $\mathrm{tr}_\omega \theta (x_0) \le C_9 \vert s  \vert^{-2}$ for some constant $C_9 > 0$. This implies that at $x_0$, one has
   \[
    \beta = \vert \nabla v\vert^2_{\omega} \le  C_9 \vert s  \vert^{- 2 }  \vert \nabla v\vert^2_{\theta} \le C_8C_{9} \vert s  \vert^{- 2 } (2B - \log \vert s\vert^2)^3
\]
Set $C_{10}:=\sup_{x\in [0,1]} \big[\sqrt x(2B-\log x)^3\big]$. Then at the point $x_0$ where $\alpha$ achieves his maximum,  we have
    \begin{eqnarray*}
    \alpha (x_0) &=& \log \beta(x_0) + k \log \vert s (x_0) \vert^2  - \gamma (v(x_0)) \\
    &\le &(k-  \frac 32) \log \vert s(x_0)\vert^2 +\log(C_8C_9C_{10}) +A(C_1+1)\\
    & \le & C_{11}.
 \end{eqnarray*}
    In the end, we get 
    \begin{eqnarray*}
    \log \beta  &=& \alpha -k \log \vert s\vert^2+\gamma(v)\\
    &\le & C_{11}-(k+A) \log |s|^2+AC_1
     \end{eqnarray*}
   which proves the required inequality with $N=k+A$.

   \end{proof}

\subsection{Laplacian estimate in families}

  The following result  -in the context of holomorphic families- 
  is a combination of the main results of \cite{CKNS85,GL10}.

  \begin{thm}  
  \label{thm laplacian estimates}
  In Setup~\ref{setup}, we assume that $\pi$ is smooth outside the basepoint $0\in X_0$. We fix $V \subset \overline{{\mathcal X}}$ a small neighborhood of $\partial {\mathcal X}$ not containing the singular point $0$.
Fix $\lambda \in \R$, $\omega$ a K\"ahler form on ${\mathcal X}$, 
$h \in {\mathcal C}^{\infty}(\partial X)$ 
and $ f \in {\mathcal C}^{\infty}(\overline{V})$
such that $\inf_{V} f =\sigma >0$. 
For all $t \in \mathbb D$, we let $u_t$ denote  
a smooth plurisubharmonic function in $\overline{V_t}$ such that
\[(dd^c u_t)^n=e^{\lambda u_t} f_t \omega^n_t
\text{ in } V_t
\; \text{ and } \; u_{|\partial X_t}=h_t.\]
Assume that there exists a smooth plurisubharmonic function $v_t$ in $\overline{V_t}$ such that
\[dd^c v_t \ge \e \omega_t \text{ in } V_t
\; \; \text{ and } \; \; 
  u_t \ge v_t
  \text{ in } V_t,
  \; \text{  with } v_t=u_t=h_t
\;   \text{ on } \; \partial X_t.\]
Then
\[\|\Delta u_t\|_{L^{\infty}(\partial X_t)} \le C,\]
where $C$ depends on upper bounds for 
$\|u_t\|_{{\mathcal C}^1(\overline{V_t})}$, $\|v_t\|_{{\mathcal C}^2(\overline{V_t})}$,
$\|f_t\|_{{\mathcal C}^1(\overline{V_t})}$, $\|h\|_{{\mathcal C}^4(\partial X_t)}$,
  $\e^{-1},\sigma^{-1}$.
\end{thm} 

The proof is an adaptation of \cite[Section 4]{GL10} which deals with
the case when $f$ is globally smooth on $\overline{{\mathcal X}}$.
  We sketch the arguments for the convenience of the reader.

\begin{proof}
We pick a smooth strictly plurisubharmonic non-positive function $\rho$ on $\overline {\mathcal X}$ such that  $\mathcal X = \{\rho < 0\}$ and $\partial \mathcal X = \{\rho = 0\}$.

We  proceed in several steps. For any fixed point $p \in \partial \cX$, we will choose local complex coordinates $(z_1, \cdots,z_n, \cdots, z_{N})$ centered at $p$ (i.e. $z(p) = 0$) defined in a neighborhood $\mathcal V$  of $p$  in $\C^N$ such that $z_{N}=\pi-\pi(p)$ and $\mathcal V \cap \mathcal X  =  \{ z \in \mathcal V \, ; \, z_{n+1} = \cdots = z_{N-1} = 0\}$.
Moreover we can arrange so that if  $z_1 = x_1 + i y_1, \cdots, z_n = x_n + i y_n$, $\partial \slash \partial x_n$ is the inner normal to $\partial X_t$ at the point $p$ and $\omega_t (p) = \sqrt{-1}\sum_j d z_j \wedge d \bar z_j$ and $(u_t)_{j \bar k}$ is diagonal at $p$. We set $\theta_t:=dd^cu_t$.

\bigskip

\noindent
{\bf Step 1 :} {\it Estimates of the tangential second order partial derivatives.} 

\medskip
\noindent
We consider the real coordinates near the point $p$ defined  as follows 
$$
s_{2j - 1} = x_j, \, \, s_{2j} := y_j, \, \, 1 \le j \le n-1, \quad \text{and} \quad  s_{2n -1} := y_n, s_{2 n} := x_n.
$$
 We will omit the subscript $t$ and use the usual notations for the partial derivatives  $ \phi_{\alpha} =  \phi_{s_\alpha}  = \frac{\partial \phi}{\partial s_\alpha}$ and $ \phi_{\alpha  \beta} := \frac{\partial^2 \phi}{\partial s_\alpha \partial s_\beta}$,  $\phi_{j \bar k} := \frac{\partial^2 u}{\partial z_j \partial \bar z_k}$ etc. We claim that there exists a constant $M_0 >0$ depending on $\|u_t\|_{\mathcal C^1(\partial X_t)}$, $\Vert  v_t \Vert_{C^2 (\partial X_t)}$,    $\Vert h\Vert_{C^2(\partial \mathcal X)}$ and $\|\rho\|_{\mathcal C^2(\partial X_t)}$ such that
 
 \begin{equation} \label{eq:tangentialC2}
 \vert u_{\alpha  \beta}  (p) \vert \le M_0, \quad \forall \,  \alpha, \beta < 2 n.
 \end{equation}
 
 \smallskip
 
 Indeed, since $u =  v$ on $\partial X_t$, we can write 
 $u -  v = \phi \rho$ near $\partial X_t$, where $\phi = \phi_t$ is a smooth function in a neighborhood of $\partial X_t$.  Moreover, one can assume that 
 \begin{equation}
 \label{norm}
 \rho_{x_n}|_{\partial X_t} \equiv -1 \quad \mbox{ in a neighborhood of} \,\, p
 \end{equation} 
 e.g. by passing to a Moser normal form for the real-analytic boundary $\partial X_t = \{x_n=\|z'\|^2+\sum_{\ell >0, |\alpha|,|\beta|>2} a_{\ell, \alpha, \beta}y_n^\ell z'^\alpha\bar z'^{\beta}$\}). Since $\rho = 0$ on $\partial X_t$, we get
 \begin{equation} 
 \label{eq;phi}
 \phi(p) =   -   (u- v)_{x_n} (p).
  \end{equation}
Similarly, we have for any $\alpha, \beta \le 2 n$
 $$
 (u-v)_{\alpha  \beta}  (p) = \phi_\alpha(p) \rho_\beta (p) + \phi_\beta (p) \rho_\alpha (p) + \phi _{\alpha \beta}  (p).
 $$
 Moreover   for any $\alpha, \beta < 2 n$, we have
 $$
  (u - v)_{\alpha  \beta}  (p) = \phi (p)  \, \rho_{\alpha \beta} (p) =  -   (u- v)_{x_n} (p)  \rho_{\alpha \beta} (p).
 $$
 Therefore the estimates \eqref{eq:tangentialC2} follow the boundary $\mathcal C^1$ estimates for $u$.
 
\bigskip

\noindent
{\bf Step 2 : }{\it Estimates of tangential-normal second order partial derivatives. }

\medskip
\noindent
These estimates rely on  the Maximum Principle for the linearized operator (see \cite[Lemma 4.3]{GL10}). With the same notations as before, for $\delta > 0$ we define 
\[ X_t (\delta) := X_t \cap B (p,\delta),\]
where $B (p,\delta)\subset \C^N$ is the euclidean  ball centered at the point $p \in \partial X_t$ with radius $\delta$. We will choose $\delta > 0$ so small that $X_t (\delta)  \subset V_t$ for any $t \neq 0$.

We will use a barrier function. For  $\tau > 0$ so that $\tau \cdot \|\rho\|_{\mathcal C^2(V_t)} < \varepsilon \slash 4$, we define
\begin{equation}
\gamma_t := (u_t -  v_t) - \tau \rho  - N \rho^2,
\end{equation}
 where $N >0$ is large enough  so that 
 $$
 n N^{1 \slash n}  \, \varepsilon  >  4 (1 + n + \varepsilon)  \sup_{V_t  \times [-M,M]} \psi_t^{1 \slash n},
 $$ 
 where $\psi_t (z,u)  := e^{ \lambda u} f_t(z) $ and $\delta$ is so small that $\delta < \tau \slash N$. Here $M := \|u_t\|_{L^{\infty}(V_t)}$.
 Under these conditions the barrier function $\gamma= \gamma_t$ satisfies the following properties (see  \cite[Lemma 4.1]{GL10}) 
 \begin{equation}
 \label{ineq gamma}
 \begin{cases}
 \Delta_\theta \gamma  \le -  \frac \varepsilon  4 \cdot  \left(1 + \mathrm{tr}_\theta \omega \right) & \text{in} \, \, X_t(\delta),\\
  \gamma \ge 0 &  \text{in} \, \, \partial X_t(\delta).
 \end{cases}
 \end{equation}
Using these estimates we can easily derive the following form of the Maximum Principle  (see \cite[Lemma 4.3]{GL10}).
 
 \begin{lem} 
 \label{lem:MaxP} Let $w \in \mathcal C^2 (\overline{ X_t(\delta)}).$ Suppose that $w$ satisfies 
 \[ \Delta_\theta w \ge -C_1(1+\mathrm{tr}_\theta\omega)  \, \, \text{on} \, \, X_t(\delta),\]
 and 
\[  w \le C_0 r^2 \, \, \text{in} \, \, B (p,\delta) \cap \partial X_t, \, \, w (p) = 0,\]
 where $r^2:= d (\cdot, p)^2 = \vert \cdot - p\vert^2$.
 
 Then the normal derivative of $w$ with respect to the inner unit normal vector $\nu$ at $p$ satisfies the inequality 
 $$  D_{\nu} w (p) = w_{x_n} (p) \le C_2,  $$
  for a constant $C_2 > 0$ depending on $\varepsilon^{-1}$, $C_0, C_1$,    $\Vert  u_t \Vert_{C^1(\partial X_t)}$, $ \Vert  v_t \Vert_{C^1(\partial X_t)}$, $\sup_{ X_t(\delta)} w$ and the constants $N, t , \delta$.
 \end{lem}
 
 \begin{proof}[Proof of Lemma~\ref{lem:MaxP}]
 This lemma is an easy consequence of \eqref{ineq gamma} coupled with the Maximum Principle for the second order linear elliptic operator 
 $ \Delta_\theta$ in $X_t(\delta)$ applied to the test function $ \phi := A \gamma + B r^2 - w$, where $A \gg B \gg 1$ are both large enough so that $\Delta_\theta \phi \le 0$ in $X_t(\delta)$, $\phi \ge 0$ in $\partial X_t(\delta)$ and $\phi (p) = 0$. Indeed, it follows from \eqref{ineq gamma} that
 \[
 \begin{cases}
 \Delta_\theta \phi \le (C_1+B\mathrm{tr}_\omega \omega_{\C^N}-\frac{A\ep}4) (1+\mathrm{tr}_\theta \omega) & \text{in} \, \, X_t(\delta),\\
 \phi \ge (B-C_0)r^2&  \text{in} \, \, B (p,\delta) \cap \partial X_t,\\
 \phi \ge B\delta^2-\sup_{ X_t(\delta)} w&  \text{in} \, \, \partial B (p,\delta) \cap X_t.
 \end{cases}
 \]
 so that $B:= \max\{C_0, \delta^{-2}\sup_{ X_t(\delta)}w\}$ and $A:=4\ep^{-1}(C_1+B\mathrm{tr}_\omega \omega_{\C^N})$ are suitable. Then by the Maximum Principle we have $\phi \ge 0$ in $X_t(\delta)$ and then $D_{\nu} \phi (p) \ge 0$. Hence $D_{\nu} w (p) \le A D_{\nu} \gamma (p) = A D_\nu (u-v) (p) - \tau D_\nu \rho  (p) \le C_2$.
\end{proof}
 
 We next apply Lemma \ref{lem:MaxP} to deduce the following estimates 
 \begin{equation} \label{eq:tangent-normalEst}
 \vert u_{s_\alpha x_n} (p)\vert \le M_1,  \quad \forall \,  \alpha < 2n,
 \end{equation}
 where $M_1 > 0$ is a constant under control.
 
 In order to do that, one considers for any $\alpha <2n$ the vector field $\mathcal T=\mathcal T_\alpha$ defined near $p$  by the formula 
\[ \mathcal T := D_\alpha - \eta D_{ x_n}\, \, \text{where} \, \, \eta := \frac{ \rho_\alpha}{ \rho_{x_n}} \, \, \text{and} \, \, \alpha < 2 n.\]
Since $\mathcal T\rho=0$, it is a tangential operator (i.e. it cancels any function that vanishes on the boundary). 
  Following \cite[pp.~1205-1206]{GL10}) one applies Lemma  \ref{lem:MaxP} to the function
 $$
 w := ( u_{y_n} - h_{y_n})^2 \pm \mathcal T (u-h), \, \, 
 $$
 to deduce an upper bound on its normal derivative at $p$ i.e.
 $$
  w_{x_n} (p) \le C,
 $$
 where $C > 0$ depends on  $\varepsilon^{-1}$,  $\Vert  f_t\Vert_{\mathcal C^1 (\bar V_t)}$, $ \Vert w\Vert_{\mathcal C^0 (V_t)}$, $ \Vert v_t\Vert_{\mathcal C^1 (\bar V_t)}$, $\Vert \eta \Vert_{\mathcal C^1(X_t(\delta))}$. Observe here that 
 $ \Vert w\Vert_{\mathcal C^0 (V\cap \bar X_t)}$ can be estimated in terms of   $ \Vert u_t\Vert_{\mathcal C^1 (V_t)}$ and $ \Vert h\Vert_{\mathcal C^1 (V_t)}$.
 From the tangential estimates \eqref{eq:tangentialC2}, we deduce a uniform upper bound on $\pm D_\nu \mathcal T (u-h)$ since $ (D_{y_n} u - D_{y_n}h) (p) = 0$. The estimates \eqref{eq:tangent-normalEst} follow with a constant $M_1$ depending on $C$ as well as $\Vert h \Vert_{C^2(\partial V_t)}$.

 Summarizing the above estimates we have proved so far that 
\begin{equation} \label{eq:finalEst}
\max_{\alpha, \beta < 2 n}\left\{\vert u_{\alpha \beta } (p)\vert,  \vert u_{\alpha x_n }  (p) \vert \right\}\le M_2, 
\end{equation}
where $M_2 = \max\{M_0,M_1\}$ depends on  $ \Vert u\Vert_{\cC^1 (V_t)}$, $  \| v \|_{C^1 (V_t)}$,  $\Vert v \Vert_{\cC^2 (\partial X_t)}$,    $\Vert h \Vert_{\cC^2(\partial X_t)}$, $\varepsilon^{-1}$. 

\bigskip

\noindent
{\bf Step 3 :}  {\it Estimates of the normal-normal second partial derivative}

\medskip
\noindent
We want to estimate $\vert u_{x_n x_n} (p) \vert$. Since 
$$
0 \le  u_{n \bar n}  (p) = \frac 14 (u_{x_n x_n} (p) +   u_{y_n y_n} (p)),
$$
it is enough to estimate $u_{n \bar n} (p)$ from above since $ u_{y_n y_n} (p)$ is uniformly bounded by  \eqref{eq:finalEst}. This is the tricky part of the proof (see \cite[Proposition 4.4]{GL10}). \\

Expanding $\mathrm{det} (u_{j \bar k} (p))$ we obtain the formula 
$$
\mathrm{det}  \left((u_{j \bar k} (p))\right) = a u_{n \bar n} (p) + b,
$$
where 
$$
 a := \mathrm{det}  \left(u_{j \bar k}  (p)\right)_{j,k \le n - 1},
$$
and $b$ is bounded by a constant depending only on the constant $M_2$ from \eqref{eq:finalEst}. So we only have to estimate $a$ from below, given our control on $\|f_t\|_{\cC^0(V_t)}$, which follows from the result below. 

\begin{lem} 
\label{positivite}
There exists a positive $c_0 > 0$ such that for any $\xi \in T_p^{1,0} \partial X_t \simeq \C^{n-1}$, we have $\|\xi\|^2_\theta \ge c_0 \|\xi\|^2_\omega$. Equivalently, 
$$
\sum_{\alpha, \beta \le n - 1} \xi_{\alpha} {\bar \xi}_\beta u_{\alpha \bar \beta} (p) \ge c_0 \vert \xi\vert^2, \xi \in \C^{n-1}.
$$
\end{lem}

\begin{proof}[Proof of Lemma~\ref{positivite}]
Let $T^\C \partial X_t \subset T^\C X_t|_{\partial X_t}$ be the complex tangent bundle to $\partial X_t$ and set
\[T^{1,0} \partial X_t =  T^{1,0} X_t \cap T_\C \partial X_t = \{\xi \in T^{1,0} X_t ; \partial \rho (\xi) = 0\}.\]
In local coordinates near $p$, we have 
\[T^{1,0} \partial X_t = \{\xi = \sum \xi_j \frac{\partial }{\partial z_j} \in T^{1,0} X_t ; \sum_j \xi_j \frac{\partial \rho }{\partial z_j} = 0\}.\]
Let us define
\[m_0 := \min \left\{ \|\xi\|^2_\theta \, ; \, \xi \in T^{1,0} \partial X_t , \| \xi \|^2_\omega = 1\right\},\]
where the minimum is taken over $\overline{X_t(\delta)}\cap \partial X_t$. Up to shrinking $\delta$, we can assume that the minimum is attained at an interior point $q\in \partial X_t\cap B(p,\delta)$ and choose a new set of coordinates $(z_i)$ centered at $q$ as before (i.e. such that $\frac{\partial}{\partial x_n}$ is the inner normal vector at $q$ such that $\rho_{x_n}|_{\partial X_t}\equiv-1$ near $q$) such that 
\begin{equation}
\label{m0}
m_0 =\theta_q(\frac{\partial}{\partial z_1}, \frac{\partial}{\partial \bar z_1})= u_{1 \bar 1}(q).
\end{equation}
 By \eqref{eq;phi}, we have
\[u_{1 \bar 1} (q) = v_{1 \bar 1} (q) - (u-v)_{x_n} (q)   \rho_{1 \bar 1} (q)\]

\begin{claim}
\label{eq:UpperBound} There exists a constant $C>0$ such that 
\[u_{x_n x_n}(q) \le C\]
and the constant $C$ depends only on $\varepsilon^{-1}$, $C_0, C_1$,    $\Vert  u_t \Vert_{\cC^1(\partial X_t)}$, $ \Vert  v_t \Vert_{\cC^1(\partial X_t)}$, $\Vert h \Vert_{\cC^4(\bar X_t(\delta))}$, the constants $N, \tau , \delta$ and a uniform bound on $\|\rho\|_{\cC^4(V_t)}$. 
\end{claim}

\begin{proof}[Proof of Claim~\ref{eq:UpperBound}]
To prove this we will apply Lemma \ref{lem:MaxP}  to a well chosen test function. For fixed $\delta > 0$ small enough, define the  following vector field 
\[\xi := - \rho_{z_n} \frac{\partial}{\partial z_1} + \rho_{z_1} \frac{\partial}{\partial z_n} \in T^{1,0}  X_t\] 
and observe that 
\[\kappa := \vert \xi \vert_\om := \left\vert - \rho_{z_n} \frac{\partial}{\partial z_1} + \rho_{z_1} \frac{\partial}{\partial z_n}\right\vert_\om > 0, \, \text{in} \, \, X_t \cap B(p,\delta)\]
since $\rho_{z_n}(q)\neq 0$. Then $\zeta := \kappa^{-1} \xi $ is a unit vector field in $ T^{1,0} X_t$ defined on $ X_t \cap B(p,\delta)$ such that $\zeta (q) \in T^{1,0}_q \partial X_t$ and the minimum $ m_0 = u_{1 \bar 1} (q)$ is achieved at $\zeta (q)$.\\

Define the following smooth function
\begin{eqnarray*}
\Phi &:=& \langle \zeta, \zeta\rangle_{dd^c h}-(u-h)_{x_n} \|\zeta\|^2_{dd^c \rho}- u_{1 \bar 1} (q)\\
&=& \sum_{j,k} h_{j \bar k} \zeta_j \bar \zeta_k - (u-h)_{x_n}  \sum_{j,k}  \rho_{j \bar k} \zeta_j \bar \zeta_k - u_{1 \bar 1} (q),
\end{eqnarray*}
on $X_t \cap B(p,\delta)$. Similarly to earlier, the identity $(u - h)|_{\partial X_t} \equiv 0$ ensures that we can write $u-h = \psi \rho$. By \eqref{norm}, we know that 
\begin{equation}
\label{psi id}
\psi|_{\partial X_t}=-(u-h)_{x_n}.
\end{equation}
Moreover, we have the following identity of forms on $\partial X_t$:
\[\partial \bar\partial (u-h)=\psi \partial \bar\partial \rho+2\Re(\partial \rho \wedge \bar \partial \psi)\]
hence  for any $\nu \in T^{1,0} \partial X_t = T^{1,0}X_t\cap \mathrm{ker}(\partial \rho)$,
\begin{equation*}
 \langle \nu, \nu\rangle_{dd^c h}=\|\nu\|^2_\theta-\psi \|\nu\|^2_{dd^c \rho}=\|\nu\|^2_\theta+(u-h)_{x_n} \|\nu\|^2_{dd^c \rho}
 \end{equation*}
 thanks to \eqref{psi id}. 
  Hence since $\zeta (q) \in T^{1,0}_q \partial X_t$, by definition of $m_0 := u_{1 \bar 1} (q)$, we have
$$
\Phi = \|\zeta\|^2_\theta  - u_{1 \bar 1} (q) \ge 0, \, \, \text{in} \, \,  B(p,\delta) \cap \partial X_t 
$$
and $\Phi(q) = 0$.

A simple computation shows that the function
$w := (u_{y_n} - h_{y_n})^2 -  \Phi $ satisfies the conditions of Lemma \ref{lem:MaxP} which implies that $w_{x_n} \le C$ and then  $\Phi_{x_n} (q) \ge  -C'$, 
where $C'> 0$ is a constant depending on $\varepsilon^{-1}$,   $\Vert  u_t \Vert_{\cC^1(\partial X_t)}$, $ \Vert  v_t \Vert_{\cC^1(\partial X_t)}$, $\Vert h \Vert_{\cC^4(\bar X_t(\delta))}$ and the constants $N, t , \delta$.

\smallskip

This implies that at $q$ we have $u_{x_n x_n}  \rho_{1 \bar 1}  \le C'',$ where $C''$ depends on $C'$ as well as $\Vert h \Vert_{\cC^3(\bar X_t(\delta))}$ and $\|\rho\|_{\cC^3((\bar X_t(\delta))}$. The claim now follows from the uniform positivity of $dd^c\rho$. 
\end{proof}

In view of \eqref{eq:finalEst} and \eqref{eq:UpperBound}, it follows that we have an a priori upper bound on all  the eigenvalues of the matrix 
$(u_{j \bar k}  (q))$.  

Since $\mathrm{det} (u_{j \bar k}  (q)) = e^{\lambda u (q)} f_t (q) \ge e^{\min_{\partial X_t} \lambda u} \sigma$, the eigenvalues of $\theta_q$ must have a uniform lower bound. In particular, $u_{1\bar 1}(q)$ is bounded away from $0$ uniformly. Given \eqref{m0}, this proves Lemma~\ref{positivite}.
\end{proof}
The theorem is now proved. 
  \end{proof}

    \bibliographystyle{smfalpha}
\bibliography{biblio}

\end{document}